\documentclass[12pt,twoside,letterpaper]{article}

\usepackage{amsmath}
\usepackage{amssymb}
\usepackage{amsfonts}
\usepackage{amsthm}
\usepackage{verbatim}
\usepackage{subfig}
\usepackage{graphicx}

\def\p{\partial}

\def\R{{ \mathbb{R}}}
\def\N{{\mathbb{N}}}
\def\Z{{\mathbb{Z}}}

\def\C{{\mathbb{C}}}

\def\di{\diamond}
\def\di{\diamond}

\def\br{{\bf r}}

\newtheorem{theorem}{Theorem}[section]
\newtheorem{lemma}[theorem]{Lemma}
\newtheorem{proposition}[theorem]{Proposition}

\newtheorem{corollary}[theorem]{Corollary}

\theoremstyle{definition}
\newtheorem{definition}[theorem]{Definition}
\newtheorem{example}[theorem]{Example}

\newtheorem{remark}[theorem]{Remark}

\newtheorem*{remark*}{Remark}

\numberwithin{figure}{section} \numberwithin{equation}{section}

\title{
Comparison of spectral radii and Collatz-Wielandt numbers for homogeneous maps,
and other applications of the monotone companion norm on ordered normed vector spaces}

\author{Horst R. Thieme\thanks{({\tt hthieme@asu.edu})}
\\
 School of Mathematical and Statistical Sciences
 \\
Arizona State University, Tempe, AZ 85287-1804, USA
}

\begin{document}
\date{March 7, 2014}
\smallskip

\maketitle

\begin{abstract}

It is well known that an  ordered normed vector space $X$ with normal cone $X_+$ has
an order-preserving norm that is equivalent to the original norm. Such an
equivalent order-preserving norm is given by
\begin{equation}
\sharp x \sharp = \max \{ d(x, X_+), d(x, - X_+)\}, \qquad x \in X.
\end{equation}
This paper explores the properties of this norm and of the
half-norm $\psi(x) = d(x,-X_+)$ independently of whether or not
the cone is normal. We use $\psi$ to derive comparison
principles for the solutions of abstract integral equations,
derive conditions for  point-dissipativity of nonlinear
positive maps,
compare Collatz-Wielandt numbers, bounds, and order spectral radii
for bounded homogeneous maps and
give conditions for a local upper Collatz-Wielandt radius
to have a lower positive eigenvector.
\end{abstract}



\section{Introduction and expos\'e of concepts and results}
\label{sec:intro}


For models in the biological, social, or economic sciences, there
is a natural  interest in solutions that are positive in an
appropriate sense, i.e., they take their values in the cone
of an ordered normed vector space.

\subsection{Cones and their properties}
\label{subsec-cone-expose}

 A closed subset $X_+$ of a normed real  vector space $X$ is called a
 {\em  wedge} if

\begin{itemize}

\item[(i)] $X_+$ is convex,

\item[(ii)] $\alpha x \in X_+$ whenever $x \in X_+$ and $\alpha \in \R_+$.

\end{itemize}

A  wedge is called a {\em cone} if

\begin{itemize}
\item[(iii)] $X_+ \cap (-X_+) = \{0\}$.
\end{itemize}

Nonzero points in a cone or wedge are called {\em positive}.

A wedge is called {\em solid} if it contains interior points.

A wedge is called {\em generating}  if
\begin{equation}
\label{eq:reproducing}
X = X_+ - X_+ ,
\end{equation}
and {\em total} if $X$ is the closure of $X_+- X_+$.

We introduce the properties of cones that will be needed
in this expos\'e. Details and more properties will be discussed in
Section \ref{sec:cones}.

A  cone $X_+$ is called {\em normal}, if
there exists some $\delta > 0$ such that
\begin{equation}
\label{eq:normal}
 \|x +z \| \ge \delta  \hbox{ whenever } x \in X_+, z \in X_+, \|x\|=1 = \|z\|.
\end{equation}
Equivalent conditions for a cone to be normal are given in Theorem
\ref{re:cone-normal}.

 $X_+$ is called an {\em $\inf$-semilattice}
 \cite{AGN} (or {\em minihedral} \cite{Kra})
 if $ x \land y = \inf \{x,y\}$
exist for all $x, y \in X_+$.

$X_+$ is called a {\em $\sup$-semilattice}
if  $x \lor y = \sup\{x,y\}$ exist for all $x,y \in X_+$.

$X_+$ is called a {\em lattice}
if $x \land y$ and $x \lor y $ exist
for all $x,y \in X_+$.

$X$ is called a {\em lattice} if $x \lor y$ exist
for all $x,y \in X$.

Since $x \land y = - ((-x) \lor (-y))$, also $x \land y$
exist for all $x,y$ in a lattice $X$.

In function spaces, typical cones are  formed by the nonnegative functions.

\bigskip

If $X_+$ is a cone in $X$, we introduce a partial order on $X$ by
$x \le y $ if $y-x \in X_+$ for $x,y \in X$ and call $X$ an {\em ordered
normed vector space}.

\bigskip

Every ordered normed vector space carries the monotone companion half-norm
$\psi(x) = d(x, -X_+)$,
\[
\psi(x) \le \psi(y), \qquad x,y \in X, x \le y,
\]
and the monotone companion norm
$\sharp x \sharp = \max \{ \psi(x), \psi(-x)\}$.
See \cite{ArCeKa}, {\cite[(4.2)]{KrLiSo}}, {\cite[L.4.1]{Bon62}}. The cone $X_+$ is normal
if and only if the monotone companion norm is equivalent to the original
norm. By the open mapping theorem, $X$ cannot be complete with respect
to both norms unless $X_+$ is normal.
The properties of  the companion (half) norm are
studied in Section 3.

\subsection{Positive and order-preserving maps}

Throughout this
paper, let $X$ be an ordered normed vector space with cone $X_+$.
We use the notation
\[
\dot X = X \setminus \{0\} \quad \hbox{ and } \quad \dot X_+= X_+ \setminus \{0\}.
\]

\begin{definition}
\label{def:order-pres}
Let $X$ and $Z$ be ordered vector space with cones $X_+$ and $Z_+$ and  $ U \subseteq X$.

Let $B: U \to Z$. $B$ is called {\em positive} if $B(U \cap X_+) \subseteq Y_+$.

$B$ is called
{\em order-preserving}  (or monotone or increasing) if $B x \le By$ whenever $x,y \in U$ and $x \le y$.
\end{definition}

Positive linear maps from $X$ to $Y$ are order-preserving.
In the following, we describe some applications of the monotone companion
(half-) norm. Other applications can be found in \cite{ArCeKa}.

\subsection{Positivity of solutions to abstract integral inequalities}

We consider integral inequalities of the following kind on an interval
$[0,b]$, $0 < b < \infty$,
\begin{equation}
\label{eq:int-ineq-intro}
u(t) \ge \int_0^t K(t,s) u(s) ds, \qquad t \in [0,b].
\end{equation}
Here $u:[0,b] \to X$ is a continuous function, $K(t,s)$, $ 0 \le s \le t \le b$,
are bounded linear positive operators such that, for each $x \in X$,
$K(t,s) x$ is a continuous function of $(t,s)$, $0 \le s \le t \le b$.
The monotone companion half-norm makes it possible to prove the following
result without assuming that the cone $X_+$ is normal (Section \ref{sec:integral}).

\begin{theorem}
Let $X$ be an ordered Banach space.
Let $u: [0,b] \to X$ be a continuous solution of the inequality (\ref{eq:int-ineq-intro}).
Then $u(t) \in X_+$ for all $ t \in [0,b]$.
\end{theorem}


\subsection{Towards a global compact attractor for nonlinear positive maps}

The monotone companion half-norm can also be useful for normal cones.
We use it to prove the following result which is important for discrete
dynamical systems on cones (Section \ref{sec:attractor}).

\begin{theorem}
\label{re:attractor}
Let $X_+$ be a normal cone and $F : X_+ \to X_+$.
Assume that there exist real constants  $c, \tilde c > 0$ and a linear, bounded, positive map $A:X \to X$ with spectral radius $\br(A) < 1$ and these properties:
\begin{equation*}
\begin{split}
& \hbox{
For any } x \in X_+ \hbox{ with } \|x\| \ge c \hbox{ there exists some }y \in X
\\
& \hbox{ with } \|y\|  \le \tilde c \hbox{ and }F(x) \le A(x) + y.
\end{split}
\end{equation*}
Then, for any bounded subset $B$ of $X_+$ there exists a bounded subset
$\tilde B$ of $X_+$ such that $F^n(B) \subseteq \tilde B$ for all $n \in \N$.

Further, $F$ is point-dissipative: There exists some $\hat c>0$ such that
\[
\limsup_{n\rightarrow \infty} \|F^n(x)\| \le \hat c,\ \ x\in X_+.
\]
If, in addition, $F$ is continuous and power compact
(or, more generally, asymptotically smooth \cite[Def.2.25]{SmTh}), then
there exists a compact subset $K$ of $X_+$ such that $F(K) =K$
and, for every bounded subset $B$ of $X_+$,  $d (F^n(x), K) \to 0$
uniformly for $x\in B$.
\end{theorem}

$F$ is power compact if $F^m$ is a compact map for some $m \in \N$.


\subsection{Spectral radii for homogeneous, bounded, order-preserving maps}


For a linear bounded  map $B$ on a complex Banach space,
the spectral radius of $B$ is defined as
\begin{equation}
\label{eq:specrad-spec}
\br(B) = \sup \{|\lambda |; \lambda \in \sigma (B)\},
\end{equation}
where $\sigma (B)$ is the spectrum of $B$,
\begin{equation}
\sigma(B) = \C \setminus \rho (B),
\end{equation}
and $\rho(B)$ the resolvent set of $B$, i.e., the set of those
$\lambda \in \C$ for which $\lambda - B$ has a bounded everywhere
defined inverse. The following alternative formula holds,
\begin{equation}
\label{eq:specrad-norm}
\br(B) = \inf_{n\in \N} \|B^n\|^{1/n} = \lim_{n\to \infty} \|B^n\|^{1/n},
\end{equation}
which is also meaningful in a real Banach space. If $B$ is a compact linear map
on a complex Banach space
and $\br(B) >0$, then there exists some $\lambda \in \sigma(B)$ and $v \in X$ such
that  $|\lambda| = \br (B)$ and $Bv = \lambda v\ne0$. Such an $\lambda $ is called
an eigenvalue of $B$. This raises the question whether $\br(B)$  could be an eigenvalue itself. There is a positive answer, if $B$ is a positive operator
and satisfies some generalized compactness assumption.

Positive linear maps are order-preserving.
They have the remarkable property that their spectral radius is a spectral
value \cite{Bon58} \cite[App.2.2]{Sch} if $X$ is a Banach space and
$X_+$ a normal generating cone.

The celebrated Krein-Rutman theorem \cite{KrRu},
which generalizes parts of the Perron-Frobenius theorem to
infinite dimensions,
establishes that a compact positive linear map
$B$ with $\br(B) > 0$ on an ordered Banach space $X$ with total  cone $X_+$ has an
eigenvector
$v \in \dot X_+$ such that $Bv = \br(B)v$ and a
positive bounded linear eigenfunctional
$v^*: X \to \R$, $v^* \ne 0$, such that $v^* \circ B = \br(B) v^*$.

This theorem has been generalized into various directions by Bonsall \cite{Bon55}
and Birkhoff \cite{Bir}, Nussbaum \cite{Nus81, Nus},
and Eveson and Nussbaum \cite{EvNu}
 (see these papers for additional references).


\subsubsection{Homogenous maps}

In the following, $X$, $Y$ and $Z$ are ordered normed vector spaces
with cones $X_+$, $Y_+$ and $Z_+$ respectively.

\begin{definition}
 $B:X_+ \to Y$ is called {\em  (positively) homogeneous (of degree one)},
 if  $B(\alpha x) = \alpha B (x)$
for all $\alpha \in \R_+$, $x \in X_+$.
\end{definition}

Since we do not consider maps that are homogeneous in other ways,
we will simply call them homogeneous maps.
If follows from the definition that
\[
B(0) = 0.
\]
Homogeneous maps are not Frechet differentiable at 0 unless $B(x+y) = B(x) + B(y)$
for all $x,y \in X_+$. For the following holds.

\begin{proposition}
Let $B: X_+ \to Y$ be homogeneous. Then the directional derivatives of $B$
exist at 0 in all directions of the cone and
\[
\p B (0,x) = \lim_{t \to 0+}\frac{ B(t x) - B(0)}{t} = B(x), \qquad x \in X_+.
\]
\end{proposition}

There are good reasons to consider homogeneous maps. Here is a mathematical one.

\begin{theorem} Let $F: X_+ \to Y$ and $u \in X$. Assume that
the directional derivatives of $F $ at $u$ exist in all directions of
the cone. Then the map $B: X_+ \to Y_+$, $B = \p F(u, \cdot)$,
\[
B (x) = \p F(u,x) = \lim_{t \to 0+} \frac{F(u+ t x) - F(u)}{t}, \qquad x \in X_+,
\]
is homogeneous.
\end{theorem}

\begin{proof}
Let $\alpha \in \R_+$. Obviously, if $\alpha=0$, $B(\alpha x) =0 = \alpha B(x)$.
So we assume $\alpha \in (0,\infty)$. Then
\[
\frac{F (u + t [\alpha x]) - F(u)}{t}
=
\alpha \frac{F(u +  [t\alpha] x) - F(u)}{t\alpha } .
\]
As $t \to 0$, also $\alpha t \to 0$ and so the directional derivative
in direction $\alpha x$ exists and
\[
\p F (u , \alpha x) = \alpha F(u,x). \qedhere
\]
\end{proof}

Another good reason are mathematical population models that take into
account that, for many species, reproduction involves a mating process
between two sexes.
The maps involved therein are not only homogeneous but also
order-preserving.

Actually, the eigenvector problem  for homogeneous maps is quite different
from the one for linear maps. Consider the following simple two sex
population model $x_n = B x_{n-1}$ where $x_n = (f_n, m_n)$ and
$B x = (B_1 x, B_2 x)$
\begin{equation}
\label{eq:two-sex-model}
\begin{split}
B_1 (f,m) = & p_f f + \beta_f \frac{f m }{f + m},
\\
B_2 (f,m) = & p_m m + \beta_m \frac{f m }{f + m}.
\end{split}
\end{equation}
This system models a population of females and males
which reproduce once a year with $f_n$ and $m_n$
representing the number of females and males at the
beginning of year $n$. The numbers $p_f, p_m$ are
the respective probabilities of surviving one year.
The harmonic mean describes the mating process and
the parameters $\beta_f$ and $\beta_m$ scale
with the resulting amount of offspring per mated pair.

$B$ always has the eigenvectors $(1,0)$ and $(0,1)$ associated
with $p_f$ and $p_m$ respectively. Looking for a different eigenvalue, $\lambda$, of $B$, we can assume that $f +m =1$ and obtain
\[
\frac{\lambda - p_f}{\beta_f} = m,
\qquad
\frac{\lambda - p_m}{\beta_m} =f,
\qquad m+f =1.
\]
We add and solve for $\lambda$,
\[
\lambda = \frac{\beta_f \beta_m + \beta_m p_f + \beta_f p_m }{\beta_f + \beta_m}
\]
and then for $m$ and $f$,
\[
m =  \frac{\beta_m  + p_m -  p_f }{\beta_f + \beta_m},
\qquad
f =  \frac{\beta_f   +  p_f -  p_m }{\beta_f + \beta_m}.
\]
The eigenvector $(m,f)$ lies in the biological relevant
positive quadrant if and only if the eigenvalue $\lambda$ is
larger than the two other eigenvalues $p_f$ and $p_m$.

We notice that we generically have  three linearly independent eigenvectors
(rather than at most two as for a $2 \times 2$ matrix)
with the third being biological relevant if and only if it is
associated with the largest eigenvalue.

The spectral radius of a positive linear map has gained
considerable notoriety because of its relation to
the basic reproduction number of population
models which have a highly dimensional structure but implicitly
assume a one to one sex ratio \cite{BaDa, CuZh, DiHeMe,Thi09, DrWa}. A spectral radius
for homogeneous order-preserving maps should play
a similar role as an extinction versus persistence  threshold parameter for
structured populations with two sexes \cite{JiTh}.

\subsubsection{Cone norms for homogeneous bounded maps}

For a homogeneous map $B:X_+ \to Y$, we define
\begin{equation}
\label{eq:operator-norm}
\|B\|_+ = \sup \{ \|Bx \|; x \in X_+, \|x\| \le 1 \}
\end{equation}
and call $B$ {\em bounded} if this supremum is a real number.
Since $B$ is  homogeneous,
\begin{equation}
\label{eq:operator-norm-est}
\|Bx \| \le \|B\|_+\, \|x\| , \qquad x \in X_+.
\end{equation}

Let $H(X_+,Y)$ denote the set of bounded homogeneous maps $B:X_+ \to Y$
and $H(X_+,Y_+)$ denote the set of bounded homogeneous maps $B:X_+ \to Y_+$
  and {\rm HM}$(X_+, Y_+)$ the set of those maps in $H(X_+,Y_+)$
  that are also order-preserving.

 $H(X_+, Y)$ is a real vector space
and $\|\cdot \|_+$ is a norm on $H(X_+,Y)$ called the cone-norm.

$H(X_+,Y_+)$ and {HM}$(X_+, Y_+)$ are cones in $H(X_+,Y)$.
We write $H(X_+) = H(X_+, X_+)$ and HM$(X_+) = $HM$(X_+,X_+)$.

It follows for  $B \in H(X_+, Y_+)$ and
 $C\in H(Y_+, Z_+)$ that
$C B \in H(X_+, Z_+)$ and
\[
\|  C B\|_+ \le \|C\|_+\, \|B\|_+.
\]

\subsubsection{Cone and orbital spectral radius}

Let $B \in H(X_+)$ and define $\phi: \Z_+ \to \R$ by $\phi (n) = \ln \|B^n \|_+$.
Then $\phi (m+n) \le \phi(m) + \phi (n)$ for all $m,n \in \Z_+$,
and a well-known result implies the following formula for the
{\em cone spectral radius}
\begin{equation}
\label{eq:cone-spec-rad-intro}
\br_+(B):= \inf_{n \in \N} \|B^n\|_+^{1/n} = \lim_{n \to \infty} \|B^n\|_+^{1/n},
\end{equation}
which is analogous to (\ref{eq:specrad-norm}).
Mallet-Paret and Nussbaum \cite{MPNu02, MPNu} suggest an alternative definition of
a spectral radius for  homogeneous (not necessarily bounded) maps
$B : X_+ \to X_+$. First, define  asymptotic least upper bounds
for the geometric growth factors of $B$-orbits,
\begin{equation}
\label{eq:growth-factor-intro}
\gamma(x,B) := \gamma_B(x) := \limsup_{n\to \infty} \|B^n (x) \|^{1/n}, \qquad x \in X_+,
\end{equation}
and then
\begin{equation}
\label{eq:spec-rad-orb-intro}
\br_o(B) = \sup_{x \in X_+} \gamma_B (x).
\end{equation}
Here $\gamma_B(x) := \infty$ if the sequence $(\|B^n (x) \|^{1/n})$
is unbounded and $\br_o(B) = \infty$ if $\gamma_B(x) = \infty$
for some $x \in X_+$ or the set $\{\gamma_B(x); x \in X_+\}$
is unbounded.

The number $\br_+(B)$ has been called {\em partial spectral radius}
by Bonsall \cite{Bon58},
$X_+$ spectral radius by Schaefer \cite{Sch59,Sch}, and
{\em cone spectral radius} by Nussbaum \cite{Nus}.
Mallet-Paret and Nussbaum \cite{MPNu02, MPNu} call $\br_+(B)$
the {\em Bonsall cone spectral radius} and $\br_o(B)$  the cone spectral radius.
For $x \in X_+$, the number $\gamma_B(x)$ has been called {\em local spectral radius} of $B$
at $x $ by F\"orster and Nagy \cite{FoNa}.

We will
follow Nussbaum's older terminology which  shares the spirit  with Schaefer's
\cite{Sch59} term
$X_+$ {\em spectral radius} and
stick with {\em cone spectral radius} for $\br_+(B)$
and call $\br_o(B)$ the {\em orbital spectral radius} of $B$.
Later, we will also introduce a {\em Collatz-Wielandt} spectral radius.
We refer  to all of them as {\em order spectral radii}.

One readily checks that
\begin{equation}
\br_+(\alpha B) = \alpha \br_+ (B), \quad \alpha \in \R_+,
\qquad
\br_+ (B^m ) = (\br_+ (B))^m , \qquad m \in \N.
\end{equation}
The same properties hold for $\br_o(B)$ though proving the
second property takes some more effort \cite[Prop.2.1]{MPNu02}.
Actually, as we show in Section \ref{sec:geometric}
for bounded $B$,
\begin{equation}
\label{eq:geom-fac-power}
\gamma(x, B^m) = (\gamma(x,B))^m, \qquad m \in \N, x \in X_+,
\end{equation}
which readily implies
\begin{equation}
\label{eq:orbital-power}
\br_o(B) = (\br_o(B))^m, \qquad m \in \N.
\end{equation}

The cone spectral
radius and the orbital spectral radius are meaningful if $B$ is just positively homogeneous and bounded,
but as in \cite{MPNu02, MPNu} we will be mainly interested in the case
that $B$ is also order-preserving and continuous.

Though the two concepts  coincide for many practical purposes,
they are both useful.

\begin{theorem}
\label{re:spec-rad-alt-char}
Let $X$ be an ordered normed vector with  cone $X_+$
and $B: X_+ \to X_+$ be continuous, homogeneous and
order-preserving.

Then $\br_+(B)\ge \br_o(B) \ge \gamma_B(x)$, $x \in X_+$.

\smallskip

\noindent
Further $\br_o(B) = \br_+(B)$ if one of the following
hold:
\begin{itemize}
\item[(i)] $X_+$ is complete and normal.

\item[(ii)]  $B$ is power compact.

\item[(iii)] $X_+$ is normal and a power of $B$ is uniformly order-bounded.

\item[(iv)] $X_+$ is complete and $B$ is additive ($B(x+y) = B(x) + B(y)$ for all
$x \in X_+$).
\end{itemize}
\end{theorem}

The inequality is a straightforward consequence of the respective definitions.
For the concepts and the proof of (iii) see Section \ref{sec:order-bounded}
and Theorem \ref{re:Col-Wie-companion2}.
The other three conditions for equality have been verified in \cite[Sec.2]{MPNu02},
(the overall assumption
of \cite{MPNu02} that $X$ is a Banach space is not used
in the proofs.) Statement (i) also follows from Theorem \ref{re:spec-rad-cone-orb}.

For a bounded positive linear operator on an ordered Banach space,
the spectral radius and cone spectral radius coincide provided that
the cone is generating \cite[Thm.2.14]{MPNu}. This is not true if the cone is only total
\cite[Sec.2,8]{Bon58}.

\subsubsection{Lower Collatz-Wielandt numbers}

For $x \in \dot X_+$,
the {\em lower Collatz-Wielandt number} of $x$ is defined as \cite{FoNa}
\begin{equation}
\label{eq:Col-Wie-num-low-intro}
[B]_x =  \sup \{\lambda \ge 0; Bx \ge \lambda x \}.
\end{equation}
By (\ref{eq:order-min-form}), $[B]_x$ is a lower eigenvalue,
\begin{equation}
B(x) \ge [B]_x x , \qquad x \in \dot X_+.
\end{equation}

\noindent
The {\em lower local Collatz-Wielandt  radius} of $x$ is defined
as
\begin{equation}
\label{eq:spec-rad-Col-Wie-loc-intro}
\eta_x(B) = :\sup_{n\in \N} [B^n]_x^{1/n} .
\end{equation}
This implies
\begin{equation}
\label{eq:loc-CW-powers-intro}
\eta_x (B^n) \le (\eta_x(B))^n, \qquad n \in \N.
\end{equation}
The {\em lower Collatz-Wielandt bound} is defined as
\begin{equation}
\label{eq:low-CW-bound}
cw(B) = \sup_{x \in \dot X_+} [B]_x,
\end{equation}
and
the {\em Collatz-Wielandt  radius} of $B$ is defined
as
\begin{equation}
\label{eq:spec-rad-Col-Wie-intro}
\br_{cw} (B) =\sup_{ x \in \dot X_+} \eta_x(B).
\end{equation}
From the definitions, $cw(B) \le \br_{cw}(B)$ and, by (\ref{eq:loc-CW-powers-intro}),
 $\br_{cw}(B^n) \le (\br_{cw}(B))^n$ for all $n \in \N$.

A homogeneous, bounded, order-preserving map $B:X_+ \to X_+$
is also bounded with respect to the monotone companion norm
and $\sharp B \sharp_+ \le \|B\|_+$. See Section \ref{sec:order-pre-companion}.
So we can define the  companion cone spectral radius, $\br_+^\sharp (B)$,
and the  companion orbital spectral radius, $\br_o^\sharp (B)$,
in full analogy to (\ref{eq:cone-spec-rad-intro}) and (\ref{eq:spec-rad-orb-intro}).
If $X_+$ is normal (and the original and the companion norm
are equivalent), the companion radii coincide with the original ones.

The monotonicity of the companion norm makes it possible to connect
the Collatz-Wielandt radius and the other spectral radii by inequalities
(Section \ref{sec:CW-lower}).

\begin{theorem}
\label{re:inequalities-radii}
Let $B:X_+ \to X_+$ be homogeneous, bounded, and order-preserving. Then
\[
cw(B) \le \br_{cw}(B) \le \br_o^\sharp(B) \le
\left\{
\begin{array}{c} \br_o(B) \\ \br_+^\sharp(B)
\end{array} \right \}
\le \br_+(B).
\]
If $X_+$ is complete and $B$ is continuous (all with respect to the original
norm), then $\br_+^\sharp(B) \le \br_o(B) $.
\end{theorem}

If $B$ is also compact, the following progressively more general results
have been proved over the years which finally lead to an extension
of the Krein-Rutman theorem from linear to homogeneous maps.

\begin{theorem}
\label{re:eigen-hist}
Let $B$ be compact, continuous, homogeneous and order-preserving and
let $r$ be any of the numbers $cw(B),  \br_{cw} (B), \br_+(B)$.
Then, if $r >0$, there exists some $v \in \dot X_+$ with $B( v) =r v$.
More specifically,
\begin{center}
\begin{tabular}{rcl}
 $r = cw(B)$ & {\em :} & {\em  Krein-Rutman 1948 \cite{KrRu}},
\\
$r = \br_{cw}(B)$ & {\em :} & {\em  Krasnosel'skii 1964 \cite[Thm.2.5]{Kra}},
\\
$r = \br_+(B)$ & {\em : }& {\em Nussbaum 1981 \cite{Nus81},
Lemmens Nussbaum \cite{LeNu13}}.
\end{tabular}
\end{center}
Actually, if $r =\br_+(B)$, then $cw(B) = \br_{cw}(B) =\br_+(B)$.
\end{theorem}

A proof for $r=cw(B)$ can also been found in \cite{Bon62}.
\cite[Thm.2.5]{Kra} is only formulated for the case that $B$ is
defined and linear on $X$, but the proof also works under
the assumptions made above. The case $r= \br_{cw}(B)$
can also been found in \cite[Cor.2.1]{Nus81}. The case $r=\br_+(B)$
is basically proved in \cite[Thm.2.1]{Nus81}, but some finishing touches
are contained in the introduction of \cite{LeNu13}. If $B(v) =r v$ with $v \in \dot
X_+$ and $r =\br_+(B)$, then $r \le [B]_v \le cw(B)$ and equality of all three numbers
follows.

It is well-known that the compactness of $B$ can be substantially relaxed
though not completely dropped if $B$ is linear \cite{Nus81}. This is also possible
(to a lesser degree)
if $B$ is just homogeneous using  homogeneous measures of noncompactness
\cite{MPNu02, MPNu}. We only mention two special cases of \cite[Thm.3.1]{MPNu02}
and \cite[Thm.4.9]{MPNu}, respectively.

\begin{theorem}
\label{re:eigen-noncomp-special}
Let $X_+$ be complete. Let $B:X_+ \to X_+$ and  $B = K + H$ where $K:X_+ \to X_+$ is homogeneous, continuous,
order preserving and compact and $H: X_+ \to X_+$ is homogeneous,
continuous, and order preserving.

Then there exists some $v \in \dot X_+$ with $B(v) = \br_+(B) v$
if $\br_+(B) > 0$ and one of the two following conditions is
satisfied in addition:

\begin{itemize}

\item[(a)] $H$ is Lipschitz continuous on $X_+$, $\|H(x) - H(y)\| \le
\Lambda \|x - y \|$ for all $x,y \in X_+$, with $\Lambda < \br_+(B)$.

\item[(b)] $X_+$ is normal, $H$ is cone-additive $(H(x+y ) = H(x) + H(y)$
for all $x,y \in X_+)$, and  $\br_+(H) < \br_+ (B)$.

\end{itemize}
\end{theorem}

The following  observation is worth mentioning
\cite[Thm.9.3]{KrLiSo} \cite[Thm.2.2]{Nus}.

\begin{proposition}
\label{re:eigen-power-self}
Let $B: X_+ \to X_+$ be {\em positively linear}, i.e., $B(\alpha x +y) = \alpha B(x) + B(y)$
for all $x,y \in X_+$ and $\alpha \ge 0$. If $r > 0$, $w \in \dot X_+$, $m\in \N$
and $B^m (w) =r^m  w$, then $B(v) =r v$ for some $v \in \dot X_+$.
\end{proposition}
Simply set
\begin{equation}
\label{eq:eigen-power-self}
v = \sum_{j=10}^{m-1} r^{-j} B^j (w).
\end{equation}
It seems to be an open problem (with a beer barrel scent \cite{MPNu-beer, Tro})
whether any of the results in Theorem
\ref{re:eigen-hist} holds if $B^2$ or some higher power of $B$ is compact
rather than $B$ itself. For additional conditions to make such a result hold,
 see \cite[Thm.7.3]{AGN}, conditions (i) and (iii),  and \cite[Sec.7]{ThiArX}.

The construction in (\ref{eq:eigen-power-self}) can be modified
to yield the following result.

\begin{proposition}
\label{re:subeigen-power-self}
Let $X_+$ be a sup-semilattice and $B:X_+\to X_+$ be homogeneous and order-preserving.
If $r > 0$, $w \in \dot X_+$, $m\in \N$
and $B^m (w) \ge r^m  w$, then $B(v) \ge r v$ for some $v \in \dot X_+$.
\end{proposition}
This time, choose $v$ as in the proof of \cite[Thm.5.1]{AGN},
\begin{equation}
\label{eq:subeigen-power-self}
v = \sup_{j=0}^{m-1} r^{-j} B^j (w).
\end{equation}
Notice that the number $r$ in Proposition \ref{re:subeigen-power-self}
satisfies $r \le [B]_v \le cw(B)$.

This observation provides conditions for equality to hold in
Theorem \ref{re:inequalities-radii}.
In order not to burden our representation with technical language
or a list of various special cases and to include possible further
developments, we make the following definition.

\begin{definition}
\label{def:KRprop}
A homogeneous bounded map $B: X_+ \to X_+$ has the {\em KR property}  (the Krein-Rutman
property) if there is some $v\in \dot X_+$ with $B(v) = \br_+(B) v$
whenever $\br_+(B)>0$.

$B$ has the {\em lower KR property} if there is some $v\in \dot X_+$ with $B(v) \ge \br_+(B) v$
whenever $\br_+(B)>0$.
\end{definition}

The maps in Theorem \ref{re:eigen-noncomp-special} are examples
that satisfy the KR property while power-compact, homogeneous,
order-preserving, continuous maps on sup-semilattices are examples that satisfy
the lower KR property (Proposition \ref{re:subeigen-power-self}).
$B:X_+ \to X_+$ is called {\em power-compact} if $B^m$ is compact
for some $m \in \N$.

 We emphasize
the lower KR property in addition to the KR property because
of our interest in population dynamics. If $B$ is the homogeneous
approximation of an nonlinear map $F$ at 0,  the condition $\br_+(B) < 1$
is often enough to guarantee the local (and some sometime the global)
 stability of 0 for the
dynamical system induced by $F$ whether or not $\br_+(B)$ is an eigenvalue
associated with a positive eigenvector. To prove persistence of the
dynamical system, it is very helpful to have a positive lower eigenvector
$B(v) \ge r v$ with $r > 1$ \cite{JiTh}.

The lower KR property turns some of the inequalities in Theorem
\ref{re:inequalities-radii} into equalities.

\begin{theorem}
\label{re:almost-all-equal}
Let $B$ be  bounded, homogeneous, order-preserving and $B^m$
have the lower KR property  for some $m \in \N$.
 Then
\[
\br_{cw} (B) =  \br_o^\sharp (B)
=
\br_+^\sharp (B) = \br_o (B) = \br_+(B).
\]
If $X_+$ is a sub-semilattice or $m=1$, then also $cw(B) = \br_+(B)$
and $B$ has the lower KR property.
\end{theorem}

The significance of $cw(B) = \br_+(B)$ is that the lower Collatz-Wielandt
numbers $[B]_x$ provide lower estimates for $\br_+(B)$ that
get arbitrarily sharp by choosing $x \in \dot X_+$ in the right way. See
Section \ref{sec:model} for some crude attempts in that direction
for a rank-structured population model with mating.

\begin{proof}
The equalities follow from Theorem \ref{re:inequalities-radii} if $\br_+(B) =0$.
So we can assume that $\br_+(B) >0$.
If $B$ has the lower KR property, it follows immediately from the definitions
that $\br_+(B) \le cw(B)$ and equality follows from Theorem \ref{re:inequalities-radii}.
If $B^m$ has the lower KR property,
\[
(\br_+(B)^m = \br_+(B^m) \le cw(B^m) \le \br_{cw}(B^m) \le (\br_{cw}(B) )^m,
\]
and the equalities follow again from Theorem \ref{re:inequalities-radii}.
If $r= \br_+(B)>0$, there exists some $w \in \dot X_+$ with $B^m(w) \ge r^m w$.
By Proposition \ref{re:subeigen-power-self}, $B(v) \ge rv$ for some $v \in
\dot X_+$ and $r \ge cw(B)$.
\end{proof}

\smallskip

The following monotone dependence of the various spectral radii on the map
can be shown without assuming normality of the cone.

\begin{theorem}
\label{re:spec-rad-increasing2}
Let $A,B: X_+ \to X_+$ be bounded and  homogeneous.
Assume that $ A(x) \le B (x) $ for all $x\in X_+$.

If $A$
is order-preserving and has the lower KR property,
then $\br_+(A) = {cw}(A) \le {cw}(B)$.

\end{theorem}


\subsubsection{Upper Collatz-Wielandt numbers}

For $x \in \dot X_+$, the upper {\em Collatz-Wielandt numbers} of $x$ \cite{FoNa} is defined as
\[
\|B\|_x = \inf \{ \lambda \ge 0; B(x) \le \lambda x\},
\]
with the convention that $\inf \emptyset = \infty$, and
a {\em local upper Collatz-Wielandt radius} $\eta^u(B)$ at $u \in \dot X_+$,
\[
\eta^u(B) = \inf_{n \in \N} \|B^n\|_u^{1/n}.
\]

{\em Collatz-Wielandt numbers} \cite{FoNa}, without  this name,
became more widely known when Wielandt used them for a new proof of the Perron-Frobenius
theorem \cite{Wie}. We will use them closer to Collatz' original purpose,
namely to prove  inclusion theorems
(Einschlie\ss ungss\"atze) for $\br_+(B)$
 which generalize those in \cite{Col1, Col2}.

If the cone $X_+$ is solid, with nonempty interior $\breve X_+$,
one can define  an {\em upper Collatz-Wielandt bound} (cf. \cite[Sec.7]{AGN})
by
\[
CW (B) = \inf_{x \in \breve X_+} \|B\|_x.
\]
These new numbers, bound, and radius relate to the former bounds and radii
in the following way and show that upper Collatz-Wielandt numbers
taken at interior points
provide upper estimates of the cone spectral radius. For this expos\'e,
we illustrate our results for the special case of a solid cone.

\begin{theorem}
\label{re:upperCW-solid}
Let $X_+$ be solid with nonempty interior $\breve X_+$
and $B: X_+ \to X_+$ be homogeneous, bounded and order-preserving. Then,
for all $u \in \breve X_+$,
\[
\gamma_B(u) = \lim_{n\to \infty} \|B^n(u)\|^{1/n},
\qquad
\eta^u(B) = \lim_{n\to \infty} \|B^n\|_u^{1/n}
\]
and
\[
cw(B) \le \br_{cw} \le \br_+^\sharp (B) \le \eta^u(B)
\le \begin{cases} \gamma_u(B) \\  CW(B) \end{cases} .
\]
We obtain some equalities in the following cases.
\begin{itemize}

\item If $X_+$ is normal, $\br_+(B )= \gamma_u(B) = \eta^u(B) \le CW(B)$
for all $u \in \breve X_+$.

\item If $B^m$ has the lower KR property for some $m \in \N$,
$\br_{cw}(B) = \br_+(B )= \gamma_u(B) = \eta^u(B) \le CW(B)$
for all $u \in \breve X_+$.

\item If $B$ has the lower KR property,
$cw(B) = \br_{cw}(B) = \br_+(B )= \gamma_u(B) = \eta^u(B) \le CW(B)$
for all $u \in \breve X_+$.

\end{itemize}
\end{theorem}

Notice that the equality $\br_+(B) = \gamma_u(B)= \lim_{n\to \infty} \|B^n(u)\|^{1/n}$ means that the cone spectral radius can be
determined by following the growth factors of an arbitrarily
chosen $u \in \breve X_+$.

The estimates from above by upper Collatz-Wielandt numbers $\|B\|_u$
with $u \in \breve X_+$ can
be arbitrarily sharp by appropriate choice of $u$ if the map is compact.
The next result should be compared to \cite[Thm.7.3]{AGN}.
Notice that $CW(B)$ here is $cw(B)$ in \cite[Thm.7.3]{AGN}.
Use of the companion half-norm makes it possible to drop
the normality of the cone.

\begin{theorem}
\label{re:upperCW-eigen-solid}
Let $X_+$ be solid and complete.
Let
$B:X_+ \to X_+$ and $B = K +A$
where $K:X_+ \to X_+$ is homogeneous, continuous,
order preserving and compact and $A: X \to X$ is linear, positive,
and bounded and $\br(A) < \br_+(B)$.

Then $cw(B) = \br_{cw}(B) = \br_+(B) = \gamma_B(u) = \eta^u(B) = CW(B)$
for all $u \in \breve X_+$.

If $r = CW(B) > 0$, then there exists some $v \in \dot X_+$ such that
$B(v) =r v$.
\end{theorem}

In Section \ref{sec:CWupper},
we prove more general versions of these theorems
replacing solidity of the cone by
uniform order-boundedness of the map.

 Using
the monotone companion half-norm,  we show
 for certain classes of homogeneous maps $B$ for which it is not
 clear whether they have the lower KR property
 that there is some $v \in \dot X_+$  such
that $B v \ge r v $ for $r = \eta^u(B)$ provided
that $r >0$ (Section \ref{sec:mono-comp}).

In Section \ref{sec:model},
 we illustrate some of our results in a discrete model for a rank-structured
population with mating.

A  version of this report will be published
under a slightly different title \cite{Thi-Pos}.


\section{More about cones}
\label{sec:cones}


We start with some more cone properties.

The cone $X_+$ is called {\em regular} if any decreasing sequence
in $X_+$ converges.

$X_+$ is called {\em fully regular} if any increasing
bounded sequence in $X_+$ converges.

The norm of $X$ is called {\em additive} on $X_+$ if $\|x + z\|= \|x\| +\|z\|$
for all $x, z \in X_+$. If the norm is additive, then the cone $X_+$ is normal.

\subsection{Normal cones}

The following result is well-known \cite[Sec.1.2]{Kra}.

\begin{theorem}
\label{re:cone-normal}
The following three properties are equivalent:
\begin{itemize}

\item[(i)] $X_+$ is normal:
There exists some $\delta > 0$ such that $\|x +z \|\ge \delta$
whenever $x \in X_+$, $z \in X_+$ and $\| x \| =1 = \|z\|$.

\item[(ii)] The norm is semi-monotonic: There exists some $M \ge 0$
such that $\|x\| \le M \|x + z\|$ for all $x,z \in X_+$.

\item[(iii)] There exists some $\tilde M \ge 0$ such that
$\|x\| \le \tilde M \|y\|$ whenever $x \in X$, $y \in X_+$,
and $- y \le x \le y$.

\end{itemize}
\end{theorem}

\begin{remark} If $X_+$ were just a wedge, property (iii) would be
rewritten as
\begin{quote}
There exists some $\tilde M \ge 0$ such that
$\|x\| \le \tilde M \|y\|$ whenever $x \in X$, $y \in X_+$,
and $ y+x \in X_+$, $y -x \in X_+$.
\end{quote}
Notice that this property implies that $X_+$ is cone:
If $x \in X_+$ and $-x \in X_+$, then $0+x \in X_+$
and $0 -x \in X_+$ and (iii) implies $\|x\| \le \tilde M \|0\|=0$.
\end{remark}

\begin{definition}
\label{def:normal-point}
An element $u \in X_+$ is called a {\em normal point} if
the set $\{\|x\|; x \in X_+, x \le u\}$ is a bounded subset of $\R$.
\end{definition}

The following is proved in \cite[Thm.4.1]{KrLiSo}.

\begin{theorem}
\label{re:normal-cones-points}
Let $X$ be an ordered normed vector space with cone $X_+$.
If $X_+$ is a normal cone, then all elements of $X_+$ are normal points
and all sets $\{\|x\| \in X; u \le x \le v\}$ with $u,v \in X$
are bounded subsets of $\R$.

If all elements of $X_+$ are normal points and $X_+$ is complete
or fully regular,
then $X_+$ is a normal cone.
\end{theorem}

\begin{proof} The first statement is obvious. Assume that $X_+$ is not
a normal cone but complete or fully regular. Then there exist sequences $(x_n)$ and $(y_n)$ in $X_+$
such that $x_n \le y_n$  and $\|x_n\| \ge 4^{n} \|y_n\|$
for all $n \in \N$. Set $v_n = \frac{y_n}{2^n \|y_n\|}$
and $u_n = \frac{x_n}{2^n \|y_n\|}$. Then $u_n \le v_n$ and
$\|v_n\| \le 2^{-n} $ and $\|u_n\| \ge 2^n $. Since $X_+$
is complete or fully regular, the series $v = \sum_{n=1}^\infty v_n$
converges in $X_+$ and $u_n \le v $ for all $n \in \N$. So $v$ is
not a normal point.
\end{proof}

Completeness or full regularity of $X_+$ are necessary for
the normality of $X_+$ as shown by the forthcoming
Example \ref{exp:complete-normal-necess}.

Connections between completeness, normality, regularity and full regularity of
cones are spelt out in the next result.

\begin{theorem}
\label{re:regular}
 Let $X$ be an ordered normed vector space with cone $X_+$.
\begin{itemize}

\item[(a)] If $X_+$ is complete and regular, then $X_+$ is normal.

\item[(b)] If $X_+$ is complete and fully regular, then $X_+$ is normal.

\item[(c)] If $X_+$ is normal and fully regular, then $X_+$ is regular.

\item[(d)] If $X_+$ is complete and fully regular, then $X_+$ is regular.
\end{itemize}
\end{theorem}

\begin{proof} Notice that the proofs in \cite[1.5.2]{Kra} and \cite[1.5.3]{Kra} only
need completeness of $X_+$ and not of $X$.
\end{proof}

\begin{theorem}
Let $X$ be an ordered normed vector space with cone $X_+$.
If $X_+$ is complete with additive norm, then $X_+$ is fully regular.

\end{theorem}

\begin{proof}
Let $(x_n)$ be an increasing sequence in $X_+$ such that
there is some $c >0$ such that $\|x_n\| \le c$ for all $n \in N$.
Define $y_n= x_{n+1} - x_n$. Then $y_n \in X_+$ and, for $m \ge j$,
 $\sum_{k=j}^m y_n = x_{m+1} - x_j$. Since the norm is
 additive on $X_+$,
 \[
 \sum_{k=1}^m \|y_n\| = \Big \| \sum_{k=1}^m y_n \Big \|
 = \|x_{m+1} - x_1\| \le 2c, \qquad m \in \N.
 \]
This implies that $(x_n)$ is a Cauchy sequence in the complete
cone $X_+$ and thus converges.
\end{proof}

The standard cones of nonnegative functions
of the Banach spaces $L^p(\Omega)$, $1 \le p < \infty$, are
regular and completely regular, while the cones of $BC(\Omega)$, the Banach
space of bounded continuous functions,  and of $L^\infty(\Omega)$
are neither regular nor completely regular. All these cones are normal.
Forthcoming examples will present a cone that is regular, but
neither completely regular, normal, nor complete.

\subsection{An example where the cone is not normal: The space of sequences of bounded variation}


Recall the Banach sequence spaces $\ell^\infty, c, c_0$  of bounded sequences,
converging sequences and sequences converging to 0
with the supremum norm and the space $\ell^1$ of summable sequences with the sum-norm.

The subsequent example for an ordered Banach space whose cone
is not normal follows a suggestion by Wolfgang Arendt.

A sequence $(x_n)$ in $\R^\N$ is called of {\em bounded variation} if the following series
converges
\begin{equation}
|x_1 | + \sum_{n=1}^\infty |x_{n+1} - x_n| =: \|(x_n)\|_{bv} .
\end{equation}
The sequences of bounded variation form a vector space, $bv$,  over $\R$ with
$\|\cdot\|_{bv}$ being a norm called the {\em variation-norm} \cite[IV.2.8]{DS}.
Notice
\begin{equation}
\label{eq:bv-containments}
\ell^1 \subseteq bv \subseteq c,
\qquad
\left \{
\begin{array}{rl}
\|x\|_\infty \le \|x\|_{bv} , & \quad   x \in bv,
\\
 \|x\|_{bv} \le 2 \|x\|_1, & \quad x \in \ell_1.
\end{array}
\right.
\end{equation}

\begin{lemma}
\label{re:bv-banach}
$bv$ with the variation-norm is a Banach space.
\end{lemma}

This is easily seen from (\ref{eq:bv-containments}) and the fact that $\ell^\infty$
is complete under the sup-norm.

Notice that $bv$ contains all constant sequences and $\|(x_n)\|_{bv}
= \|x_1\| = \|(x_n)\|_\infty$ if $(x_n)$ is a constant sequence.
Actually, all monotone bounded nonnegative sequences are of bounded variation.

\begin{lemma}
\label{re:bv-increasing}
If $(x_n)$ is a nonnegative bounded increasing sequence, then $(x_n)$ is of bounded variation
and $\|(x_n)\|_{bv} = \|(x_n)\|_\infty = \lim_{n\to \infty} x_n$.

If $(x_n)$ is a nonnegative decreasing sequence, then $(x_n)$ is of bounded variation
and $\|(x_n)\|_{bv} = 2 x_1 - \lim_{n\to \infty} x_n$.
\end{lemma}

There are several cones we can consider in $bv$. The  one we are going to
consider here is the cone
of nonnegative sequences of bounded variation, $bv_+$. Others are the cone
of nonnegative increasing sequences and the cone of nonnegative
decreasing sequences.

\begin{proposition}
$bv_+$ is generating, solid, but not normal (and not regular and not fully
regular). $X$ is a lattice and the lattice operations are continuous: If $x= (x_j) \in bv$, then $|x|= (|x_j|) \in bv$
and $\big \| \, |x|  \big\|_{bv} \le \|x\|_{bv}$ with strict inequality being
possible.
\end{proposition}

Notice that every monotone nonnegative sequence that is bounded away from zero
is the interior of $bv_+$ as can be seen from (\ref{eq:bv-containments}).
The space of Lipschitz continuous functions with Lipschitz norm is an example
of an ordered normed vector space and lattice where the cone of nonnegative
functions is not
normal and  the lattice operations are not continuous \cite[p.535]{MPNu02}.

That $bv_+$ is not normal follows from the following result
that characterizes the normal points. Since $bv_+$
is complete, $bv_+$ is then also not regular and not fully regular (Theorem
\ref{re:regular}).

\begin{theorem}
\label{re:bv-normal-points}
$x=(x_n)\in bv_+$ is a normal point if and only if $x \in \ell_1$.
If $x \in \ell_1$, then
\[
 \|x\|_1 - (1/2)x_1 \le \sup \{\|v\|_{bv}; v \in bv_+, x-v \in bv_+\} \le 2 \|x\|_1.
\]
\end{theorem}

\begin{proof} One half of the statement is easy to see. Recall (\ref{eq:bv-containments}).
For the other part,
assume that $x=(x_n)\in bv_+$ is a normal element.
Let $y$ be the sequence $( x_1, 0, x_3, 0 , x_5 , \ldots )$. Then $0 \le y \le x$
and $y \in bv$
and $\|y\|_{bv} = |x_1| + 2 \sum_{j=1}^\infty |x_{2j+1} |$.
Further let $z$ be the sequence $(0, x_2, 0, x_4, 0 , \ldots ) $.
Then $0 \le z \le x$ and $z \in bv$ and $\|z\|_{bv} = 2 \sum_{j=1}^\infty \|x_{2j}\|$.
Hence $x \in \ell^1$ and
\[
2 \|x\|_1 - x_1 = \|y\|_{bv} + \|z\|_{bv}\le 2 \sup\{\|v\|_{bv}; 0 \le v \le x \}. \qedhere
\]
\end{proof}

\begin{example}[Wolfgang Arendt]
 \label{exp:complete-normal-necess}
The cone $\ell^1_+$ in $\ell^1$ with the variation norm is not normal
though all points in $\ell^1_+$ are normal elements. It is regular,
but not fully regular.
\end{example}

\begin{proof}
Let $x^m $ be the sequence where $x^m_j = 0$ for $j > 2m$, $x^m_j =0$
for all odd indices and $x^m_j =1$ otherwise. Let $u^m$ be the sequence
where $u^m_j = 0$ for $j > 2m$ and $u^m_j =1$ otherwise. Then $\|u^m\|_{bv}=2$
and $\|x^m\|_{bv} = 2m $. For all $m \in \N$, $x^m \le u^m$ but
$\|x^m\|_{bv} = m \|u^m\|_{bv}$. So $\ell^1_+$ is not normal under the
variation norm. Theorem \ref{re:normal-cones-points} implies that $\ell^1_+$
is not fully regular under the variation norm. Let $(x_n)$ be a decreasing
sequence in $\ell^1_+$. Since $\ell^1_+$ is regular under the sum-norm,
there exists some $x \in \ell^1_+$ with $\|x_n - x \|_1 \to 0$.
By (\ref{eq:bv-containments}), $\|x_n -x \|_{bv} \to 0$.
\end{proof}

\subsection{A convergence result}

The following convergence principle, which has been distilled from
the proof of \cite[Thm.5.2]{AGN}, will be applied several times.

\begin{proposition}
\label{re:monotone-converge}
 Let $X$ be an ordered normed vector space.
Let $S \subseteq X^\N$ be a set of sequences with terms in $X$ with the
property that with $(x_n) \in S$ also all subsequences  $(x_{n_j}) \in S$.
Assume that every increasing (decreasing) sequence $(x_n) \in S$ has a convergent subsequence.
The every increasing (decreasing) sequence  $(x_n) \in S$ converges.
\end{proposition}

\begin{proof} Let $(x_n) \in S$ be increasing (the decreasing case
is similar). Then there exist a subsequence $(x_{n_j})$
and some $x \in X$ such that $x_{n_j} \to x$. Suppose that $(x_n)$ does
not converge to $x$. Then there exist some $\epsilon > 0$ and
a strictly increasing sequence $(k_i)\in \N$ such that $\| x_{k_i} - x \|
\ge \epsilon $ for all $i \in \N$. By assumption, $(x_{k_i})\in S$;
further it is an increasing sequence. So, after choosing a subsequence,
$x_{k_i} \to y$ for some $y \ne x$. Fix $n_j$. If $i$ is sufficiently
large, $k_i \ge n_j$ and $x_{n_j} \le x_{k_i}$. Taking the limit $i \to
\infty$ yields $x_{n_j} \le y$ because the cone is closed.
Now we let $j \to \infty$ and obtain $x \le y$. By symmetry, $y \le x$,
a contradiction.
\end{proof}


\section{Local spectral radii}
\label{sec:geometric}


Let $B:X_+ \to X_+$ be homogeneous and   $x \in X_+$.  We prove (\ref{eq:geom-fac-power}),
\begin{equation}
\label{eq:eq:geom-fac-power-in}
\gamma (x,B^m) \le (\gamma(x,B))^m , \qquad m \in \N,
\end{equation}
with equality holding if $B$ is bounded.

From the properties of the limit superior,
\[
\begin{split}
\gamma(x,B^m) & = \limsup_{k \to \infty} \|(B^m)^k(x)\|^{1/k} =
\lim_{j \to \infty} \sup_{k \ge j} \|B^{mk}(x)\|^{m/(mk)}
\\
\le &
\lim_{j \to \infty} \sup_{n \ge mj} (\|(B^{n}(x)\|^{1/n})^m
=
\Big (\limsup_{n\to \infty}\|(B^{n}(x)\|^{1/n} \Big )^m= (\gamma(x,B))^m .
\end{split}
\]
To prove the opposite inequality if $B$ is bounded,
suppose that $\gamma (x,B^m) < (\gamma(x,B))^m $. Then there exists some
$s\in (0,1) $ such that
\[
\gamma (x,B^m) < s^m (\gamma(x,B))^m .
\]
By definition of $\gamma (x,B^m)$, there exists some $N \in \N$ such that
\begin{equation}
\label{eq:geometric-proof}
\|B^{mn}(x)\|^{1/(mn)} < s \gamma(x,B), \qquad n \ge N .
\end{equation}
Choose a sequence $(k_j)$ such that $k_j \to \infty$ and
$\|B^{k_j}(x)\|^{1/{k_j}} \to \gamma (x,B)$.
Then there exist sequences $n_j $ and $p_j$ such that $n_j \to \infty$
and $0 \le p_j < m$  and $k_j = m n_j  + p_j$. By (\ref{eq:operator-norm-est}),
\[
\|B^{k_j}(x)\|^{1/{k_j}} \le \|B^{p_j}\|_+^{1/k_j}\; \| B^{m n_j} (x) \|^{1/k_j}.
\]
By the properties of the limit superior,
\[
\gamma (x, B) \le  \limsup_{j\to \infty} (\|B^{p_j}\|_+^{1/k_j})
\limsup_{j \to \infty} ( \| B^{m n_j} (x) \|^{1/(mn_j)})^{(k_j - p_j)/k_j}.
\]
By (\ref{eq:geometric-proof}) and $p_j/k_j \to 0$,
\[
\begin{split}
\gamma (x, B)
\le & \limsup_{j\to \infty}  (s \gamma(x,B))^{(k_j - p_j)/k_j}
\\
\le &  \limsup_{j \to \infty} (s\gamma(x, B))^{1 - (p_j/k_j)}
= s \gamma (x, B),
\end{split}
\]
a contradiction.



\section{Monotone companion norm and half-norm}
\label{sec:companion}


Every ordered normed vector space carries an order-preserving
half-norm which we call the {\em (monotone)  companion half-norm}
(called the {\em canonical half-norm} in \cite{ArCeKa}).

\begin{proposition}[cf. \cite{ArCeKa}, {\cite[(4.2)]{KrLiSo}}, {\cite[L.4.1]{Bon62}}]
\label{re:normal-mon-funct}
 We define the (monotone) companion half-norm $\psi: X \to \R_+$ by
\begin{eqnarray}
\psi(x)& = &\inf \{ \|x +z \|; z \in X_+\}= d(x, -X_+)
\label{eq:mon-compan-funct}
\\
&= &
\inf \{ \|y\|; x \le y \in X\}, \qquad \qquad x \in X.
\label{eq:mon-compan-funct-incr}
\end{eqnarray}
Then the following hold:

\begin{itemize}

\item[(a)]
$\psi$ is positively homogenous and order-preserving on $X$.

\item[(b)] $\psi$ is subadditive on $X$
($\psi(x+y) \le \psi(x)+ \psi(y)$, $x,y \in X$),
\[
|\psi(x) - \psi(y)|\le  \;\|x-y\|, \qquad x,y \in X.
\]

\item[(c)]
For $x \in X$,  $\psi(x) = 0$ if and only if $ x \in -X_+$.
In particular $\psi $ is strictly positive: $\psi(x) > 0$ for all $x \in \dot  X_+$.

 \item[(d)] $X_+$ is normal if and only if there exists some $\delta > 0$ such that $\delta \|x\| \le \psi(x)
$ for all $x \in X_+$.

\item[(e)] If the original norm $\|\cdot\|$ is order-preserving on $X_+$, then $\|x\| = \psi(x)$
for all $x \in X_+$.
\end{itemize}
\end{proposition}

Here $d(x, -X_+)$ denotes the distance of $x$ from $-X_+$.

\begin{proof}
The functional $\psi$ inherits positive homogeneity  from
the norm. That $\psi$ is order-preserving is immediate from
(\ref{eq:mon-compan-funct-incr}). For all $x \in X$,  $x \le x$ and so $\|x\| \ge \psi (x)$.

Most of the other properties follow from (\ref{eq:mon-compan-funct})
and the assumption that $X_+$ is a cone.

Since $\psi$ is subadditive,
\[
| \psi (x) - \psi(y)| \le \psi(x-y) \le  \|x-y\|.
\]
If $-x \in X_+$, then $ x \le 0 $ and $\psi(x) \le \|0\|=0$.

Assume that $x\in X$ and  $\psi(x) =0$.
By definition, there exists a sequence $(y_n)$ in $X$ with $\|y_n\| \to 0$
and $y_n \ge x $ for all $n \in \N$. Then $y_n -x  \in X_+$. Since $X_+$
is closed, $-x = \lim_{n\to \infty} (y_n -x) \in X_+$.

The strict positivity of $\psi$ follows from $X_+ \cap (-X_+) = \{0\}$.

(d) Assume that $X_+$ is normal. By Theorem \ref{re:cone-normal} (ii),
 there exists some $c > 0$ such that
$\|x\| \le c \|y\|$ whenever $x,y \in X_+$ and $x \le y$.
Hence $\|x\| \le c \psi(x) $ for all $x \in X_+$. Set $\delta = 1/c$.

The converse follows from $\psi$ being order-preserving and $\psi(x) \le \|x\|$
for all $x\in X_+$.
\end{proof}

One readily checks that the monotone companion half-norm inherits
equivalence of norms.

\begin{remark}
\label{re:mon-compan-equiv}
 If $\|\cdot \|_\sim$ is a norm on $X$ that is equivalent
to $\|x\|$, i.e., there exists some $c > 0$ such that $(1/c) \|x\|
\le \|x\|_\sim \le c \|x\|$ for all $x \in X$, then then
respective monotone companion half-norms satisfy $(1/c) \psi(x)
\le \tilde \psi(x) \le c \psi(x)$ for all $x \in X$.
\end{remark}

The functional $\psi$ induces a monotone norm on $X$ which is equivalent
to the original norm if and only if $X_+$ is a normal cone.

\begin{theorem}[cf. {\cite[Thm.4.4]{KrLiSo}}]
\label{re:normal-equiv-norm}
Define
\[
\sharp x\sharp = \max \{ \psi(x), \psi (-x) \}, \qquad x \in X,
\]
with $\psi$ from Proposition \ref{re:normal-mon-funct}. Then
$\sharp \cdot\sharp $ is a norm on $X$ with the following properties.

\begin{itemize}

\item  $\sharp x \sharp \le \|x\|$ for all $x \in X$,
$\sharp x\sharp = \psi(x)$
if $x \in X_+$, and $\sharp x\sharp = \psi(-x)$ if $-x \in X_+$.

 \item $\sharp \cdot\sharp$ is order-preserving on $X_+$: $\sharp x\sharp \le \sharp y\sharp $ for all $x, y \in X_+$
with $x \le y$. Moreover, for all $x,y, z \in X$ with $x\le y\le z$,
\begin{equation}
\label{eq:mod-norm}
\sharp y \sharp \le \max\{ \sharp x\sharp, \sharp z\sharp \}.
\end{equation}

 \item $X_+$ is a normal cone if and only if the norm $\sharp \cdot\sharp$  is equivalent to
the original norm.

\item If the original norm $\|\cdot\|$ is order-preserving on $X_+$, then
$\|x\| = \sharp x \sharp$ for all $x \in X_+ \cup (-X_+)$.

\end{itemize}
\end{theorem}

\begin{proof} It is easy to see from the properties of $\psi$
that $\sharp \alpha x\sharp = |\alpha| \sharp x\sharp$
for all $\alpha \in \R$, $x\in X$, and that $\sharp \cdot \sharp$ is subadditive.
Now $\psi(x) \le \|x\|$ and $\psi(-x) \le \|-x\|= \|x\|$ and
so $\sharp x\sharp \le \|x\|$.

To prove (\ref{eq:mod-norm}), let $x,y,z \in X$ and $x \le y \le z$.
Then $y \le z $ and $-y \le -x$. Since $\psi$ is order-preserving on $X$,
\[
\psi(y ) \le \psi (z)  \le \sharp z\sharp,
\qquad
\psi(-y) \le \psi(-x) \le \sharp x\sharp.
\]
(\ref{eq:mod-norm}) now follows from the definition of $\psi$
as do most of the remaining assertions.

That $\|\cdot\|$ and $\sharp \cdot \sharp$ are equivalent  norms
if the cone is normal
is shown in \cite[Thm.4.4]{KrLiSo}. The converse follows from
Theorem \ref{re:cone-normal} (ii).
\end{proof}

\begin{definition}
\label{def:mon-compan}
 The norm $\sharp \cdot\sharp$ is called the
{\it (monotone) companion norm} on the ordered normed vector space $X$.
\end{definition}

\begin{corollary}
\label{re:compare}
Let  $X_+$ be a normal cone.
Then there exists some $c \ge 0$ such that $\|y \| \le c\max\{ \|x\|, \|z\|\}$
for all $x,y,z \in X$ with $x \le y \le z$.
\end{corollary}

\begin{proof} Let $\sharp \cdot\sharp$ be the monotone companion norm  from Theorem \ref{re:normal-equiv-norm}, which is equivalent to the
 original norm because $X_+$ is normal.  Choose $c \ge 0$ such
that $\sharp x\sharp \le \|x\| \le c \sharp x\sharp$ for all $x \in X$.
Let $x \le y \le z$.
By Theorem \ref{re:normal-equiv-norm},
\[
\|y\| \le c \sharp y \sharp \le c\max\{ \sharp x\sharp , \sharp z\sharp \} \le
c \max\{ \|x\| , \|z\| \}. \qedhere
\]
\end{proof}

\begin{corollary}[Squeezing theorem {\cite[Thm.4.3]{KrLiSo}}]
\label{re:squeeze}
Let  $X_+$ be a normal cone.
Let $y \in X$ and $(x_n)$, $(y_n)$, $(z_n)$ be
sequences in $X$ with $x_n \le y_n \le z_n $ for all $n \in \N$
and $x_n \to y$ and $z_n \to y$. Then $y_n \to y$.
\end{corollary}

\begin{proof} Notice that $x_n - y \le y_n -y \le z_n -y$.
By Corollary \ref{re:compare}, with some $c\ge 0$ that does not depend on $n$,
\[
\|y_n - y\| \le c  \max\{\|z_n - y\|, \|x_n -y\|\} \to 0. \qedhere
\]
\end{proof}

\begin{example} (a) Let $X =\R$ with the absolute value. Then the
monotone companion norm is also the absolute value and the companion half-norm
is the positive part,
$\psi(x) = \max\{x,0\}$ for all $x \in \R$.

(b) Let $X = \R^2$ with the maximum norm and $X_+ = \R_+^2$. Then the monotone companion norm
is also the maximum norm.

(c) Let $X = \R^2$ with $\|\cdot\|$ being either the Euclidean norm or the sum norm and $X_+= \R_+^2$.
Then $\sharp x \sharp = \|x\|$ if  $x \in X_+ \cup (-X_+)$
while $\sharp x \sharp$ is the maximum norm of $x$ otherwise.
\end{example}

\begin{example} Let $X$ be a normed vector lattice \cite[II.5]{Sch74}.
 Since $x \le x^+$, $\psi(x) \le \|x^+\|$.
Let $x \le y$. Then $x^+ \le y^+$, and $\|x^+\| \le \|y^+\|\le \|y\|$.
Hence $\|x^+\| \le \psi(x) $. In combination,
\[
\psi(x) = \|x^+\|, \qquad x \in X.
\]
Further $\psi(-x) = \|(-x)^+ \|= \|x^-\|$. So
\[
\sharp x \sharp = \max\{ \|x^+\|, \|x^-\|\}, \qquad x \in X,
\]
and $\sharp x \sharp = \|x\|$ for all $x \in X_+ \cup (-X_+)$.
If $X$ is an abstract M space \cite[II.7]{Sch74},
\[
\sharp x \sharp = \|x^+ \lor x^- \| = \big \| |x| \big \| = \|x\|.
\]
\end{example}

We turn to ordered Banach space the cone of which may be not normal.

\begin{example}
Let $X \subseteq \tilde X$ where $X, \tilde X$ are ordered normed vector spaces
with norms $\|\cdot \|$ and $\|\cdot\|^\sim$ and $\|x\| \ge \|x\|^\sim$ for
all $x \in X$. Let $\psi, \tilde \psi$ be the respective monotone companion
half-norms. Then $\psi(x) \ge \tilde \psi (x)$ for all $x \in X$.

Assume that $(\tilde X, \|\cdot\|^\sim)$  is an abstract M space. By the previous
example,
\[
\sharp x\sharp  \ge \|x\|^\sim, \qquad x \in X.
\]
Now assume that there exists some $u \in X$ with $\|u\|\le 1 $ and $|x| \le \|x\|^\sim u$
for all $x \in X$.
Then $\psi (\pm x) \le \|x\|^\sim $ and $\sharp x \sharp \le \|x\|^\sim$.
In combination,
\[
\sharp x \sharp = \| x \|^\sim, \qquad x \in X.
\]
\end{example}

\begin{example}[Space of sequences of bounded variation]
To determine the monotone companion (half-) norm in a concrete case where  the cone
is not normal we revisit the space $bv$ of sequences of bounded variation
with the variation-norm.

Recall that $bv \subseteq \ell^\infty$ and $\|x\|_{bv} \ge \|x\|_\infty$
for all $x \in bv$. $\ell^\infty$ is an abstract M space under the sup-norm.
We apply the previous example with $u$ being the sequence all the terms of which
are 1. Then $\|u\|_{bv} =1 $ and $ |x | \le \|x\|_\infty u $.
We obtain that $\sharp x \sharp = \|x\|_\infty$ for all $x \in bv$.

Notice that $bv$ is not complete under the monotone companion norm.
By the open mapping theorem, an ordered Banach space is complete
under the monotone companion norm if and only if the cone is normal.
\end{example}


\section{Order-preserving maps and the companion norm}
\label{sec:order-pre-companion}

In the following, let $X$ and $Y$ be ordered normed vector spaces.
We use the same symbols $\|\cdot\|$, $\psi$ and $\sharp \cdot \sharp$ for
the norms and the monotone companion (half-) norms on $X$ and $Y$.

\begin{theorem}
\label{re:mon-func-maps}
Let $B: X_+ \to Y_+$ be bounded, homogeneous, and order-preserving. Then
$\psi(B(x)) \le \|B\|_+ \psi(x)$ for all $x \in X_+$. In particular, $B$
is bounded with respect to the monotone companion norms and $\sharp B \sharp_+
\le \|B\|_+$.
\end{theorem}

\begin{proof}
Since $B$ is order-preserving, for $x\in X_+$,
\[
\{ y \in Y; B(x) \le y\} \supseteq \{ B(z); x \le z \in X_+\}.
\]
By definition of $\psi$,
\[
\begin{split}
\psi(B(x) ) \le & \inf \{ \|B(z)\| ; z \in X_+, x \le z \}
\le
\inf \{ \|B\|_+ \|z\| ;  x \le z \in X\}
\\
&  = \|B\|_+ \psi(x). \qedhere
\end{split}
\]
\end{proof}

\begin{theorem}
\label{re:bounded-linear-maps-compan}
Let $B: X \to Y$ be bounded, linear and positive. Then
$B$
is bounded with respect to the monotone companion norm, $\psi(Bx) \le \|B\| \psi(x)$
for all $x \in X$, and
$\sharp B \sharp \le \|B\|$.
\end{theorem}

\begin{proof}
By a similar proof as for Theorem \ref{re:mon-func-maps}, since $B$ is order-preserving,
$\psi (Bx) \le \|B \| \psi(x)$ for all $x \in X$. Then
$\psi(- B x) = \psi (B(-x)) \le \|B\| \psi (-x).$
The assertion now follows from $\sharp x \sharp = \max \{\psi(x), \psi(-x)\}$.
\end{proof}

Various concepts of continuity are preserved if one switches from the
original norm to the monotone companion norm. We only look at
the most usual concept here.

\begin{proposition}
\label{re:right-conti-mon-comp}
Let  $B: X \to Y$ be order-preserving and continuous at $x$ with respect
to the original norms.
Then $B$ is continuous at $x$ with respect to the
monotone companion norms.
\end{proposition}

\begin{proof}
Let $x\in X$ and $B$ be continuous at $x$. Let $y \in X$ as well.

For any $z \ge y-x$, $z \in X$, we have
\[
B(y) - B(x) = B (x+y-x)- B(x) \le B(x+z) - B(x).
\]
By definition of the monotone companion half-norm,
\[
\psi (B(y) - B(x) ) \le \|B(x+z) - B(x)\|, \qquad y-x \le z \in X.
\]
Let $\epsilon > 0$. Then there exists some $\delta_+ > 0$
such that $\|B(x+z) - B(x)\|< \epsilon$ for all $z\in X$,
$\|z \|< \delta_+$. Now let $y \in X$ and $\psi(y-x) < \delta_+$. Then there exists some
$z \in X$ such that $y-x \le z $ and $\|z\| < \delta_+$. Then $\|B(x+z) - B(x) \|< \epsilon$ and
$\psi (B(y) - B(x) ) < \epsilon$.

Also, for any $z \in X$ with $x- y \le z$, we have
\[
B(x) - B(y) = B(x) - B(x-(x-y)) \le B(x) - B(x-z).
\]
By definition of the monotone companion half-norm
\[
\psi (B(x) - B(y) ) \le \|B(x) - B(x-z)\|, \qquad z \in X,  x-y \le z.
\]
Let $\epsilon > 0$. Then there exists some $\delta_- > 0$
such that $\|B(x) - B(x-z)\|< \epsilon$ for all $z\in X$,
$\|z \|< \delta_-$. Now let $y \in Y$ and $\psi(x-y) < \delta_-$. Then there exists some
$z \in X$ such that $x-y \le z $ and $\|z\|=\|-z\| < \delta_-$.
Thus $\psi (B(x) - B(y) ) < \epsilon$.

Set $\delta= \min\{\delta_+, \delta_-\}$ and $\sharp y -x \sharp < \delta$.
Then $\psi(\pm(y-x)) < \delta_\pm$
and $\psi( \pm (B(y) - B(x)) < \epsilon$. This implies
$\sharp B(y) - B(x) \sharp < \epsilon$.
\end{proof}

\begin{proposition}
\label{re:bounded-func-compan}
Let $\phi: X_+ \to \R_+$ be homogeneous.
\begin{itemize}

\item[(a)] If $\phi$ is bounded with respect to the monotone companion
norm, then it is bounded with respect to the original norm and
$\| \phi \|_+ \le \sharp \phi \sharp_+$.

\item[(b)] If $\phi$ is order-preserving and bounded with respect to
the original norm, then it is bounded with respect to the
monotone companion norm and $\| \phi \|_+ = \sharp \phi \sharp_+$.
Further $\phi(x) \le \psi(x) \|\phi\|_+$ for all $x \in X_+$.
\end{itemize}
\end{proposition}

\begin{proof} (a)
Let $x \in X_+$. Since $\phi$ is bounded with respect to the
monotone companion norm,
\[
\phi (x) \le \sharp \phi \sharp_+ \; \sharp\,x \sharp
\le
\sharp \phi \sharp_+ \; \|x\|.
\]

(b) This follows from part (a) and Theorem \ref{re:mon-func-maps}.
\end{proof}

\begin{proposition}
\label{re:bounded-func-compan-lin}
 Let $\phi: X \to \R$ be linear.
\
\begin{itemize}
\item[(a)] If $\phi$ is bounded with respect to the monotone companion
norm, then it is bounded with respect to the original norm and
$\| \phi \| \le \sharp \phi \sharp$.

\item[(b)] If $\phi$ is positive and bounded with respect to
the original norm, then it is bounded with respect to the
monotone companion norm and $\| \phi \| = \sharp \phi \sharp$
and $\phi(x) \le \|\phi\| \psi(x)$ for all $x \in X$.

\end{itemize}
\end{proposition}

\begin{proof} (a) The proof is similar to the one for
Proposition \ref{re:bounded-func-compan}.

(b)  follows from part (a) and Theorem \ref{re:bounded-linear-maps-compan}
and the fact that the monotone half-norm on $\R$ is given by the positive part.
\end{proof}

Let $X^*_+$ be the dual wedge of  positive linear functionals on $X$
that are bounded with respect to the original norm $\|\cdot\|$.
By Proposition \ref{re:bounded-func-compan-lin}, $X_+^*$ is also the dual wedge of
linear functions that are bounded with respect to the monotone
companion norm $\sharp \cdot \sharp$ and the norms induced on $X^*$ are
the same on $X_+^*$.

\begin{proposition}
Let $x \in X$. Then there exists $x^* \in X_+^*$ such that
$x^* x = \psi(x)$ and $\|x^*\| = \sharp x^* \sharp \le 1$.
If $x \in X_+$, $\|x^*\|=1$ can be achieved.

Actually
\begin{equation}
\label{eq:mon-compan-dual1}
\begin{split}
\psi(x) = & \max \{ x^* x; x^* \in X_+^*, \|x^*\| \le 1\}, \qquad x \in X,
\\
\psi(x) = &  \max \{ x^* x; x^* \in X_+^*, \|x^*\| = 1\}, \qquad x \in X_+.
\end{split}
\end{equation}

Further \cite[(4.3)]{KrLiSo}, for all $x \in X$,
\begin{eqnarray}
\sharp x \sharp = & \max \{|x^* x|; x \in X_+^*, \|x^*\|= 1 \}
= \max \{|x^* x|; x \in X_+^*, \|x^*\|\le 1 \}.
\label{eq:mon-compan-dual4}
\end{eqnarray}
\end{proposition}

\begin{proof}
Recall that $X_+^*$ and its norm do not depend on whether we
consider $X$ with the original norm or its monotone companion norm
(Proposition \ref{re:bounded-func-compan-lin}).

Let $u \in X $.
By \cite[IV.6]{Yos}, we find $x^* \in X^*$ with
$x^* u  = \psi(u)$ and
\[
-\psi(-x) \le x^* x \le \psi (x) , \qquad x \in X,
\]
and so
\[
|x^* x | \le  \sharp x \sharp  \le \|x\|, \qquad x \in X.
\]
 Since $\psi(-x)=0$ for all $x\in X_+$, $x^* \in X_+^*$.
 By
Proposition \ref{re:bounded-func-compan-lin}, $\sharp x^* \sharp =
\|x^*\| \le 1$.

Now let $x \in X_+$. Then $x^*x = \psi(u) = \sharp u \sharp$ and so
$\sharp x^* \sharp =1$.

Notice that we have proved $\le $ in (\ref{eq:mon-compan-dual1}).

Suppose that $\psi(u) > \psi(-u)$.
Then there exists $x^* \in X_+^*$ with $\sharp x^* \sharp \le 1$
such that $x^* u = \psi (u) = \sharp u \sharp$. Hence $\sharp x^*\sharp =1 $.

If $\psi(u) < \psi (-u)$, there exists $x^* \in X_+^*$ with $\sharp x^* \sharp \le 1$
such that $x^* (-u) = \sharp u \sharp= |x^* u|$. Again $\sharp x^* \sharp =1$.

If $\psi(u) = \psi(-u) > 0$, then $u ,-u \not\in X_+$ and we can make
the same conclusion.

It remains the case
 $\psi(u) = \psi(-u) = 0$. But then $u=0$, and all equalities hold trivially.

So the $\le $ inequalities hold in  (\ref{eq:mon-compan-dual4}).

The $\ge$ inequalities follow from Proposition \ref{re:bounded-func-compan-lin}
(b).
\end{proof}

\begin{corollary}
\label{re:cone-closed-compan}
$X_+$ is closed with respect to the monotone
companion norm.
\end{corollary}

\begin{proof} Let $(x_n)$ be a sequence in $X_+$, $x \in X$ and
$\sharp x_n - x \sharp \to 0$. Suppose that $x \notin X_+$.
By a theorem of Mazur \cite[IV.6.Thm.3']{Yos}, there exists a bounded linear $\phi: X \to \R$
such that $\phi(x) < -1$ and $\phi(y) \ge -1 $ for all $y \in X_+$.
Let $z \in X_+$, $n \in \N$. Then $nz \in X_+$ and $\phi(nz) \ge -1$.
So $\phi(z) \ge -1/n$. We let $n \to \infty$ and obtain $\phi(z) \ge 0$.
By Proposition \ref{re:bounded-func-compan-lin}, $\phi$ is continuous
with respect to the monotone companion norm. So $0 \le \phi(x_n) \to \phi(x)$
and $\phi(x) \ge 0$, a contradiction.
\end{proof}

\section{Positivity of solutions to abstract integral inequalities}
\label{sec:integral}

We consider integral inequalities of the following kind on an interval
$[0,b]$, $0 < b < \infty$,
\begin{equation}
\label{eq:int-ineq}
u(t) \ge \int_0^t K(t,s) u(s) ds, \qquad t \in [0,b].
\end{equation}
Here $u:[0,b) \to X$ is a continuous function, $K(t,s)$, $ 0 \le s \le t \le b$,
are bounded linear positive operators such that, for each $x \in X$,
$K(t,s) x$ is a continuous function of $(t,s)$, $0 \le s \le t \le b$.

\begin{theorem}
Let $X$ be an ordered Banach space.
Let $u: [0,b] \to X$ be a continuous solution of the inequality (\ref{eq:int-ineq}).
Then $u(t) \in X_+$ for all $ t \in [0,b]$.
\end{theorem}

\begin{proof} We define $v: [0,b) \to X$ by $v(t) =- u(t)$. Then
\begin{equation}
\label{eq:int-ineq2}
v(t) \le \int_0^t K(t,s) v(s) ds, \qquad t \in [0,b].
\end{equation}
Since the monotone companion half-norm $\psi$ is order-preserving,
\begin{equation*}
\psi(v(t)) \le \psi \Big( \int_0^t K(t,s) v(s) ds\Big ), \qquad t \in [0,b].
\end{equation*}
Since $\psi$ is convex and homogeneous,
\begin{equation}
\label{eq:int-ineq3}
\psi(v(t)) \le  \int_0^t \psi ( K(t,s) v(s)) ds, \qquad t \in [0,b].
\end{equation}
By Theorem \ref{re:bounded-linear-maps-compan},
\[
\psi(K(t,s) u(s)) \le \|K(t,s)\| \psi(s), \qquad 0 \le s \le t \le b.
\]
By the uniform boundedness theorem, there exists some $c \ge 0$
such that $\|K(t,s)\| \le c$ whenever $0 \le s \le t \le b$. So
\begin{equation}
\label{eq:int-ineq4}
\psi (K(t,s) v(s) ) \le c \psi (v(s)), \qquad 0 \le s \le t \le b.
\end{equation}
We substitute the last inequality into (\ref{eq:int-ineq3}),
\begin{equation*}
\psi(v(t)) \le  c \int_0^t  \psi ( v(s)) ds, \qquad t \in [0,b].
\end{equation*}
Let $\lambda > 0$. Then
\begin{equation*}
e^{-\lambda t} \psi(v(t)) \le  c \int_0^t e^{-\lambda (t-s)} e^{-\lambda s} \psi ( v(s)) ds, \qquad t \in [0,b].
\end{equation*}
Define $\alpha(\lambda) = \sup_{0 \le t \le b} e^{-\lambda t} \psi(v(t))$. Then
\[
 e^{-\lambda t} \psi(v(t)) \le  c \int_0^t e^{-\lambda (t-s)} \alpha(\lambda) ds
\le
c \frac{\alpha(\lambda)}{\lambda}, \qquad 0 \le t \le b,
\]
and
\[
\alpha(\lambda) \le c\frac{\alpha(\lambda)}{\lambda}.
\]
Choosing $\lambda >0$ large enough, $\alpha (\lambda) \le 0$ and, since it is nonnegative,
$\alpha (\lambda) = 0$. This implies $\psi(v(t)) = 0$ for all $t \in [0,b]$.
By Proposition \ref{re:normal-mon-funct}, $v(t) \in -X_+$ and so $u(t)=-v(t) \in X_+$ for all $t \in [0,b]$.
\end{proof}


\section{Towards a compact attractor of bounded sets}
\label{sec:attractor}

The monotone companion half-norm can also be useful if the cone $X_+$ is normal.
The following result generalizes \cite[Prop.7.2]{SmTh}.

\begin{proposition}
\label{re:attr-exist}
Let $X_+$ be a normal cone and $F : X_+ \to X_+$.
Assume that there exist some linear, bounded, positive map $A:X \to X$,
 $\br(A) < 1$, and real numbers  $R, c > 0$ such that
\begin{equation}
\label{eq:attr-exist}
\begin{split}
& \hbox{
for any } x \in X_+ \hbox{ with } \|x\| \le R \hbox{ there exists some }y \in X
\\
& \hbox{ with } \|y\|  \le c \hbox{ and }F(x) \le A(x) + y.
\end{split}
\end{equation}
Then, after introducing an equivalent norm, there exists some $\tilde{R}>0$ such that \[
F(X_+\cap\bar{B}_s)\subset\bar{B}_s,\ \ s\ge \tilde{R},
\]
for all closed balls $\bar{B}_s$ with radius $s\ge \tilde{R}$ and the origin as center. Further, there exists some $\hat{R}>0$ such that
\[
\limsup_{n\rightarrow \infty} \|F^n(x)\| \le \hat{R},\ \ x\in X_+.
\]
\end{proposition}

\begin{proof}
Since $X_+$ is normal, the original norm is equivalent to its monotone companion
norm $\sharp \cdot \sharp$. Let $\psi$ be the monotone companion half-norm. Then
(\ref{eq:attr-exist}) can be rewritten as
\begin{equation}
\label{eq:attr-exist1}
\psi (F(x) - A(x))  \le c, \qquad x \in X_+, \|x\| \ge R.
\end{equation}
In a next step we can switch to an equivalent monotone norm,
denoted by $\|\cdot\|^\sim$, such that the associated operator norm
of $A$ satisfies $\|A\|^\sim < 1$,
\[
\|x\|^\sim =  \sum_{k=0}^m r^{-m} \sharp A^k (x) \sharp, \qquad x \in X,
\]
where $ \br_+(A) < r < 1 $ and $m \in \N$ is chosen large enough
such that $\|A^{m+1}\| < r^{m+1}$. See \cite[2.5.2]{Kra}. Then $\|A\|^\sim \le r < 1$.
Let $\tilde \psi$ be the monotone companion half-norm  of $\|\cdot\|^\sim$.
Since $\|\cdot \|^\sim$ is order-preserving on $X_+$, $\|x\|^\sim
= \tilde \psi(x)$ for $x \in X_+$.

The inequality (\ref{eq:attr-exist1})
remains valid for $\|\cdot \|^\sim $ and $\tilde \psi$
after possibly changing $R$ and $c$.

So, without loss of generality, we can assume that $\|\cdot\|$
is order-preserving on $X_+$, $\|A\| < 1$, and the monotone companion
half-norm of $\|\cdot\|$ satisfies (\ref{eq:attr-exist1}).

 Since $\psi$ is monotone and subadditive on $X$, for $x\in X_+$ with $\|x\| \ge R$,
\begin{equation}
\label{eq:dissip}
\begin{split}
\|F(x) \| = &  \psi(F(x))  \le \psi(F(x) - A x) +  \psi(Ax)
\\
\le & c + \|A x\|
\le
c + \|A\|\, \|x\|.
\end{split}
\end{equation}
Choose some $\xi \in (\|A\|,1)$. Let $\breve R > R$. Then, if $x \in X_+$ and $\|x\| \ge \breve R$,
\[
\|F(x)\| \le \xi \|x\| + c + (\|A\|-\xi) \|x\|
\le
\xi \|x\| + c + (\|A\|-\xi) \breve R.
\]
Choosing $\breve R > R$ large enough, we can achieve that
$c + (\|A\|-\xi) \breve R \le 0$ and
\begin{equation}
\label{eq:dissip1}
\|F(x)\| \le \xi \|x\|, \qquad \|x\| \ge \breve R.
\end{equation}
 Set $\tilde R = \max\{ \breve R, c + \|A\| \breve R\}$.
By (\ref{eq:dissip}) and (\ref{eq:dissip1}),
\[
\|F(x)\| \le \max \{ \tilde R, \xi \|x\| \}, \qquad x \in X_+.
\]
For any $s \ge \tilde R$,  $\|x\| \le s$ implies $\|F(x) \| \le s$.
By (\ref{eq:dissip1}) $\|F^n(x) \| \le \xi^n \|x\| $ as long as
$\|F^n(x)\|\ge \tilde R \ge \check R$. Once $\|F^m(x) \| \le \tilde R$
for some $m \in \N$, then $\|F^n(x) \| \le \tilde R$ for all $n \ge m$.
\end{proof}

\begin{proof}[Proof of Theorem \ref{re:attractor}]
Let $\| \cdot \|^\sim $ be an equivalent norm such that
there is some $\tilde R > 0$ such $F(X_+ \cap \bar B_s) \subseteq (\bar B_s)$ for all $s \ge \tilde R$ where $\bar B_s$ is the closed ball with radius $s$ and the origin as center
taken with respect to $\|\cdot \|^\sim$. Let $B$ be a bounded subset of $X_+$
with respect to the original norm. Then $B \subseteq \bar B_s$ for some $s \ge \tilde R$.
By Proposition \ref{re:attr-exist}, $F^n(B) \subseteq \bar B_s$ for all $n \in \N$. Notice that $\tilde B = B_s$
is bounded with respect to the original norm. The statement concerning
the compact attractor $K$ follows from \cite[Thm.2.33]{SmTh}.
\end{proof}

\section{The space of certain order-bounded \\ elements
and some functionals}
\label{sec:background-order-bounded}


\begin{definition}
\label{def:order-norm}
Let $x \in X$ and $u \in X_+$. Then $x$ is called $u$-{\em bounded} if there
exists some $c > 0$ such that $-cu \le x \le c u$. If $x$ is $u$-bounded, we define
\begin{equation}
\label{eq:order-max}
\|x\|_u = \inf \{c > 0; -cu \le x \le c u \}.
\end{equation}
The set of $u$-bounded elements in $X$ is denoted by $X_u$.
If $x,u \in X_+$ and $x $ is not $u$-bounded, we define
\[
\|x\|_u = \infty.
\]
Two elements $v$ and $u$ in $X_+$ are called {\em comparable} if $v$ is $u$-bounded
and $u$ is $v$-bounded, i.e., if there exist $\epsilon, c > 0$
such that $\epsilon u \le v \le c u$.
Comparability is an equivalence relation for elements of $X_+$,
and we write $u \sim v$ if $u$ and $v$ are comparable.
Notice that $X_u = X_v$ if
and only if $ u \sim v$.
\end{definition}

If $X$ is a space of real-valued functions on a set $\Omega$,
\[
\|x\|_u = \sup \Big \{ \frac{|x(\xi)|}{u(\xi)}; \xi \in \Omega, u(\xi) > 0\Big\}.
\]
Since the cone $X_+$ is closed,
\begin{equation}
\label{eq:order-max-for}
-\|x\|_u u \le x \le \|x\|_u u, \qquad x \in X_u.
\end{equation}
$X_u$ is a linear subspace of $X$,
$\|\cdot\|_u$ is a norm on $X_u$,
and $X_u$, under this norm, is an ordered normed vector space with
cone $X_+ \cap X_u$ which is normal, generating, and has nonempty
interior.

\begin{lemma}
\label{re:order-norm-prop}
 Let $u \in \dot X_+$. Then the following hold:

\begin{itemize}

\item[(a)]
$\sharp x \sharp \le \|x\|_u \; \sharp \,u \sharp= \|x\|_u \,
\max\{ d (u, -X_+), d(u,X_+)\}, \qquad x \in X_u.$

\item[(b)]  $u$ is a normal point of $X_+$ if and only if there exists some $c \ge 0$
such that $\|x\| \le c \|x\|_u $ for all $x \in X_u \cap X_+$.

\item[(c)] If $X_+$ is solid and $u$ in the interior of $X_+$,
then $X = X_u$, $d(u, X \setminus X_+)>0$, and
\[
\|x\| \ge \|x\|_u\, d(u, X \setminus X_+), \qquad x \in X.
\]

\item[(d)]
In turn, if $X_u = X$ and there exists some $\epsilon >0$
such that $\|x\| \ge \epsilon \|x\|_u$ for all $x \in X$,
then $u$ is an interior point of $X_+$.

\item[(e)] If $X_+$ is solid and $u$ is both an interior and a normal point of $X_+$,
then $\|\cdot\|$ and $\|\cdot\|_u$ are equivalent.
\end{itemize}
\end{lemma}

\begin{proof} (a) By (\ref{eq:order-max-for}), if $x \in X_u$,
\[
x, -x \le \|x\|_u u .
\]
Since the companion functional is order-preserving on $X$,
\[
\psi (x) \le \|x\|_u \psi(u), \qquad \psi(-x) \le \|x\|_u \psi(u)
\]
and so $\sharp x \sharp \le \|x\|_u \sharp u \sharp$.

(b) Let $u$ be  normal point of $X_+$. Then there exists some
$c > 0$ such that $\|y\| \le c $ for all $y \in X_+$ with
$ y \le u $. For $x \in X_+ \cap X_u$, $ x \le \|x\|_u u $.
If $x \ne 0$ in addition, $\|x\|_u^{-1} x \le u$.
Hence $\big \| \|x\|_u^{-1} x \big \|\le c$.

The other direction is obvious.l

(c) Let $u$ be an interior element of $X_+$. Then
$d(u, X \setminus X_+) >0$. For any $ \delta \in (0, d(u, X \setminus X_+)) $,
we have $u \pm \frac{\delta }{\|x\|} x \in X_+$
for all $x \in \dot X$. So $ \pm x \le \frac{\|x\|}{\delta }u $
and so  $x \in X_u$ and $\|x\|_u  \le \frac{ \|x\|}{\delta}$.
Since this holds for any $\delta \in (0, d(u, X \setminus X_+) $,
it also holds for $\delta = d(u, X \setminus X_+)$.

(d) Assume that  $X_u = X$ and there exists some $\epsilon >0$
such that $\|x\| \ge \epsilon \|x\|_u$ for all $x \in X$.
This means that
\[
\pm x \le \|x\|_u u \le (1/\epsilon) \|x\| u.
\]
Hence
$u \pm \frac{\epsilon}{\|x\|} x \in X_+$ for all $x \in \dot X$.
This implies that $u$ is an interior point of $X_+$.

(e) follows from combining (b) and (c).
\end{proof}

If $X_+$ is normal, by Theorem \ref{re:cone-normal}, there
exists some $M \ge 0$ such that
\begin{equation}
\label{eq:order-max-for-norms}
\|x\| \le M \|x\|_u \|u\|, \qquad x \in X_u.
\end{equation}
If $X_+$ is a  normal and complete cone of $X$, then
$X_+ \cap X_u$ is a complete subset of $X_u$ with the metric
induced by the  norm $\|\cdot\|_u$. For more information
see \cite[1.3]{Kra} \cite[I.4]{Boh},  \cite[1.4]{KrLiSo}.

For $u \in X_+$, one can also consider the functionals
\[
\left . \begin{array}{cc}
(x/u)^\di = & \inf\{ \alpha \in \R; x \le \alpha u \}
\\
(x/u)_\di = & \sup\{ \beta \in \R; \beta u \le x \}
\end{array} \right \} x \in X,
\]
with the convention that $\inf (\emptyset) = \infty$ and
$\sup (\emptyset) =- \infty$.
If $X$ is a space of real-valued functions on a set $\Omega$,
\[
\left . \begin{array}{cc}
(x/u)^\di = & \sup\big \{  \frac{x(\xi)}{u(\xi)}; \xi \in \Omega, u(\xi) >0 \big\}
\\[2mm]
(x/u)_\di = & \inf\big\{ \frac{x(\xi)}{ u(\xi)}; \xi \in \Omega, u(\xi) >0 \big\}
\end{array} \right \} x \in X.
\]
Many other symbols have been used for these two functionals
in the literature; see Thompson \cite{Tho} and Bauer \cite{Bau} for some
early occurrences.
For $x \in X_+$, $\|x\|_u =(x/u)^\di$.
Since we will use this functional for $x \in X_+$ only, we will
stick with the notation $\|x\|_u$.
Again for $x \in X_+$, $(x/u)_\di$ is a nonnegative real number,
and we will use the leaner notation
\begin{equation}
\label{eq:order-min}
[x]_u = \sup \{\beta \ge 0; \beta u \le x \}, \qquad x, u \in X_+.
\end{equation}
Since the cone $X_+$ is closed,
\begin{equation}
\label{eq:order-min-form}
x \ge [x]_u u, \qquad x,u \in X_+ .
\end{equation}
Further $[x]_u$ is the largest number for which this inequality holds.

\begin{lemma}
\label{re:order-min}
Let $u \in \dot X_+$. Then the functional $\phi= [\cdot]_u: X_+ \to \R_+$ is
homogeneous, order-preserving and
concave. It is  bounded with respect to the original norm on $X$
and also to the monotone companion norm,
\[
[x]_u \le \frac{\sharp x \sharp }{\sharp u \sharp}
\le
\frac{\|x\|}{\sharp u \sharp}, \qquad x \in X_+,
\]
 and
\[
 \| \phi \|_+ = \sharp  \phi \sharp_+
\le \frac{1} {\sharp\, u \sharp}.
\]
$\phi$ is upper semicontinuous with respect to the original norm and
\[
\big | [y]_u - [x]_u \big | \le \|y -x \|_u, \qquad y,x \in X_u \cap X_+.
\]
\end{lemma}

Recall that $\sharp u \sharp = d(u, - X_+)$.

\begin{proof}
We apply the  monotone companion norm to  (\ref{eq:order-min-form}),
\[
\sharp\, x \sharp \ge [x]_u \, \sharp \, u\sharp, \qquad x \in X_+.
\]
The equality $ \| \phi \|_+ = \sharp  \phi \sharp_+$
 follows from Proposition \ref{re:bounded-func-compan}.
 The other properties are readily derived from the definitions.
\end{proof}
See
\cite{Krau} for an in-depth treatment of this functional.


\section{Lower Collatz-Wielandt numbers, bounds, and  radius}
\label{sec:CW-lower}

Let $B:X_+ \to X_+$ be homogeneous and order-preserving.
We do not assume that $B$ is continuous at 0; so concepts
like the cone spectral radius \cite{LeNu12, MPNu02, MPNu, Nus} may not apply.
For $x \in \dot X_+$,
the {\em lower Collatz-Wielandt number} of $x$ is defined as \cite{FoNa}
\begin{equation}
\label{eq:Col-Wie-num-low}
[B]_x = [Bx]_x = \sup \{\lambda \ge 0; Bx \ge \lambda x \}.
\end{equation}
By (\ref{eq:order-min-form}),
\begin{equation}
Bx \ge [B]_x x , \qquad x \in \dot X_+.
\end{equation}
The following result is immediate.

\begin{lemma}
Let $B,C: X_+ \to X_+$ be homogeneous, bounded, and order-preserving.
Let $x \in \dot X_+$. Then
$[CB]_x \ge [C]_x [B]_x $.
\end{lemma}
This implies
\begin{equation}
\label{eq:col-wie-low-ineq}
b_{n+m} \ge b_n b_m , \quad b_n = [B^n]_x, \qquad n, m \in \N.
\end{equation}
The {\em lower local Collatz-Wielandt  radius} of $x$ is defined
as
\begin{equation}
\label{eq:spec-rad-Col-Wie-loc}
\eta_x(B) = \sup_{n\in \N} [B^n]_x^{1/n} .
\end{equation}
This implies
\begin{equation}
\eta_x (B^n) \le (\eta_x(B))^n, \qquad n \in \N.
\end{equation}

The {\em lower Collatz-Wielandt bound} is defined as
\begin{equation}
cw(B) = \sup_{ x \in \dot X_+} [B]_x,
\end{equation}
and
the {\em Collatz-Wielandt  radius} of $B$ is defined
as
\begin{equation}
\label{eq:spec-rad-Col-Wie}
\br_{cw} (B) =\sup_{ x \in \dot X_+} \eta_x(B).
\end{equation}

\subsection{Companion spectral radii}

If $B:X_+ \to X_+$ is homogeneous and bounded with respect to
the original norm, then, by Proposition \ref{re:mon-func-maps}, $B$ is also bounded with respect to
the monotone companion norm $\sharp \cdot \sharp$ and $\sharp B \sharp_{\,+} \le \|B\|_+$. So, we can define {\em the companion cone spectral
radius}, the {\em companion growth bounds},  and the {\em companion orbital spectral radius} by
\begin{equation}
\br_+^\sharp(B) = \inf_{n\in \N} \sharp B^n \sharp_{\,+}^{1/n}
=
\lim_{n \to \infty} \sharp B^n \sharp_{\,+}^{1/n}
\end{equation}
and
\begin{equation}
\br_o^\sharp (B) = \sup_{x \in X_+} \gamma_B^\sharp(x),
\qquad \gamma_B^\sharp(x) = \limsup_{n\to \infty} \sharp B^n x \sharp^{1/n}.
\end{equation}
Since $\sharp \,x \sharp \le \|x\| $ for all $x \in X$, we have the estimates
\begin{equation}
\label{eq:geom-growth-companion-compare}
\begin{split}
\br_o^\sharp (B) \le \br_+^\sharp(B) \le \br_+ (B),
\\
\br_o^\sharp (B) \le \br_o(B) \le \br_+ (B).
\end{split}
\end{equation}
If the cone $X_+$ is normal, the companion norm is equivalent to
the original norm and the respective spectral radii equal
their companion counterparts. The following proposition implies
Theorem \ref{re:inequalities-radii}.

\begin{proposition}
\label{re:Col-Wie-lower}
Let $B$ be bounded, homogeneous and order-preserving.
Then, for all $x \in \dot X_+$,
\[
[B]_x \le \eta_x(B) \le \gamma_B^\sharp (x) \le \gamma_B (x).
\]
Further $cw(B) \le \sharp B \sharp \le \|B\|$ and
\[
cw(B) \le \br_{cw} (B) \le  \br_o^\sharp (B)
\le
\left \{ \begin{array}{c}\br_+^\sharp (B)  \\ \br_o (B) \end{array} \right \}
\le \br_+(B).
\]
\end{proposition}

\begin{proof}
Let $x \in \dot X_B$. The first inequality follows
from (\ref{eq:spec-rad-Col-Wie-loc-intro}).
Since the cone is closed,  $B(x) \ge [B]_x x$.
By induction $B^n(x) \ge [B]_x^n x $.

We apply the monotone companion half-norm $\psi$,
which is homogeneous, and obtain
$\psi( B^n(x)) \ge [B]_x^{n} \psi(x)$.
Since $\psi(x) > 0$,
$\gamma_x(B) \ge \gamma_x^\sharp (B) \ge [B]_x$.
Since, for all $n \in \N$, $B^n$ is
homogeneous and order-preserving,
\[
[B^n]_x \le  \gamma_x^\sharp (B) \le  \gamma_x^\sharp (B)^n.
\]
The last equality follows (\ref{eq:eq:geom-fac-power-in}).
Since this holds for all $n \in \N$,
by (\ref{eq:spec-rad-Col-Wie-loc-intro})
$\eta_x(B) \le \gamma_x^\sharp(B)$.
The last inequality follows
from (\ref{eq:geom-growth-companion-compare}).

The remaining assertions follow directly from the definitions.
\end{proof}

The following criteria for the positivity of the lower Collatz-Wielandt
bound and the Collatz-Wielandt radius are obvious from their definitions.

\begin{lemma}
Let $B:X_+ \to X_+$ be homogeneous and order-preserving.

Then
$cw(B) >0$ if and only if there exist $\epsilon > 0$
and $x \in \dot X_+$
such that $ B(x) \ge \epsilon x$.

Further
$\br_{cw} (B) > 0$ if and only if there exists $\epsilon > 0$, $m \in \N$,
and $x \in \dot X_+$
such that $B^m (x) \ge \epsilon x$.
\end{lemma}

\begin{theorem}
\label{re:spec-rad-cone-orb}
Let $X$ be an ordered normed vector with complete cone $X_+$
and $B: X_+ \to X_+$ be homogeneous, continuous and
order preserving.
Then $\br_+^\sharp (B) \le \br_o(B)$.
\end{theorem}

\begin{proof}
Let $r > \br_o(B)$.
It is sufficient to show that $\br_+^\sharp(B) < r$.
By definition of $\br_o(B)$, $\gamma_B(x) < r $ for all $x \in X_+$.
Let $m \in \N$. For all $x \in X_+$, by definition of $\gamma_B(x)$, there exists some $n \in \N$,
$n \ge m$, such that $\|B^n (x)\| \le r^n$.
In other words,
\[
X_+ \subseteq  \bigcup_{n \ge m} C_n, \qquad C_n= \{x \in X_+; \|B^n ( x )\| \le  r^n \}.
\]
Since $B$ is continuous, the sets $C_n$ are closed with respect to $\|\cdot\|$.
Since $X_+$ is complete with respect to $\|\cdot\|$, by the Baire
category theorem, there exists some $n \in \N$, $n \ge m$, such that
$ C_n$
contains a  subset of $X_+$ that is relatively open with respect to $\|\cdot\|$. So there
exist some  $\epsilon > 0$ and $x \in X_+$ such that
$\|B^n (x + \epsilon z) \| \le  r^n $ whenever $z \in X$, $\|z\| \le 2$,
and  $x + \epsilon z \in X_+$. Let  $z \in X_+$, $\|z\| \le 2 $.
Since $x \in X_+$ and $B$ is order preserving,  $B^n (\epsilon z) \le B^n (x + \epsilon z)$.
Since the companion norm is monotone and $B$ is  homogeneous,
\[
 \epsilon \sharp B^n z\sharp   \le \sharp B^n (x + \epsilon z) \sharp
 \le \| B^n (x + \epsilon z) \| \le r^n, \qquad z \in X_+, \|z\|\le 2.
\]
Now let $ y \in X_+$ with $\sharp y \sharp = \psi(y) \le 1$.
By definition of the companion half-norm,  there exists some $z \in X$ with $z \ge y$ and $\|z \| \le 2$.
Then $z \in X_+$ and, since $B$ is order preserving, $ B^n(y) \le B^n(z)$. Since the
companion norm is monotone,
\[
\sharp B^n(y) \sharp \le \sharp B^n(z) \sharp
\le r^n/\epsilon.
\]
Since this holds for all $y\in X_+$ with $\sharp y \sharp \le 1$, $\sharp B^n  \sharp_+
\le r^n/\epsilon  $. Since $n \ge m$,
\[
\inf_{n\ge m} \sharp B^n\sharp_+^{1/n} \le r \inf_{n \ge m}  \epsilon^{-1/n}.
\]
Thus $\br_+^\sharp(B) = \lim_{n \to \infty} \sharp B^n\sharp_+^{1/n} \le r$.
\end{proof}

We conclude this section with a conditions under which  the lower KR property implies the KR property (see Definition \ref{def:KRprop}).
The following definition is similar to the one in \cite[III.2.1]{Boh}.

\begin{definition}
\label{def:stric-incr}
Let $\theta: X_+ \to \R_+$ be order-preserving and homogeneous.
An order-preserving map $B: X_+ \to X_+$ is called
strictly $\theta$-increasing if for any $x, y \in X_+$ with
$x \le y$ and $\theta(x) < \theta(y)$ there exists some $\epsilon > 0$
and some $m \in \N$ such that $B^m (y) \ge (1+\epsilon) B^m( x)$.

$B$ is called strictly increasing if $B$ is strictly $\theta$-increasing
where $\theta$ is the restriction of the norm to $X_+$.
\end{definition}

\begin{theorem}
\label{re:strictly-power-comp}
Let $\theta:X_+ \to X_+$
be order-preserving and homogeneous and  $B:X_+ \to X_+$
be continuous, homogeneous, and strictly $\theta$-increasing.
Assume that there is some $p \in \N$ such that $(B^p (x_n))_{n\in \N}$
has a convergent subsequence for any increasing  sequence $(x_n)$
in $X_+$ where $\{ \theta(x_n); n \in \N\}$ is bounded.
Then $B$ has the KR property whenever it has the lower KR
property.
\end{theorem}

\begin{proof} Since $B$ is homogeneous, we can assume that $\br_+(B) =1$
and that there exists some $x \in \dot X_+$
such that $B(x) \ge x$.
 Then the sequence $(x_n)$ defined
by $x_n = B^n(x)$ is increasing.
We claim that $(\theta(x_n))$ is bounded. If not, then there exists some $n \in \N$
with $\theta(x_{n-1}) <  \theta(x_n)$ where $x_0 =x$. Since $B$ is strictly
$\theta$-increasing, there exists some $\epsilon > 0$ and some
$m \in \N$ such that $B^m (x_n) \ge (1+\epsilon) B^m (x_{n-1})$.
By definition of $(x_n)$,
$B (y) \ge (1+\epsilon) y $ for $y = B^m( x_{n-1}) \ge x$.
Since $y \in \dot X_+$, $\br_+(B) \ge cw(B) \ge 1+ \epsilon$, a
contradiction.

Choose $p \in \N$ according to the assumption of the theorem.
We  apply the convergence principle in Proposition \ref{re:monotone-converge}.
Let $S$ be the set of sequences $(B^p (y_n))$ where $(y_n)$ is a
increasing sequence in $X_+$ such that $(\theta(y_n))$ is bounded.
Then $S$ has the property required in Proposition \ref{re:monotone-converge}
and so every increasing sequence in $S$ converges.

 Since $(x_n)$ is increasing and bounded,
 $(B^p(x_n))\in S$  converges  with limit $v$.
Since $B^p(x_n) = x_{n+p}$, $x_n \to v$.
 Since $x_{n+1} = B(x_n)$
and $B$ is continuous, $B(v)=v$.
\end{proof}

We mention some interesting properties of strictly $\theta$-increasing maps.

\begin{proposition}
\label{re:Perron}
Let $\theta: X_+ \to \R_+$ be order-preserving and homogeneous
and $B: X_+ \to X_+$ be homogeneous and strictly $\theta$-increasing.
Let $r,s > 0$ and $v,w \in  X_+$  with $\theta(v) >0$ and $\theta (w)>0$.

\begin{itemize}

\item[(a)] If
 $B(v) \ge r v$ and $B(w) \le s w$ and $v$ is $w$-bounded,
then $r \le s$ and $r=s$ implies $w \ge \frac{\theta(w)}{\theta(v)} v$.

\item[(b)] If
 $B(v) = r v$ and $B(w) = s w$ and $v$ and $w$ are comparable,
then    $r=s$ and $ w= \frac{\theta(w)}{\theta(v)} v$.

\end{itemize}
\end{proposition}

\begin{proof} We first assume that $\theta(v) =1= \theta (w)$.

(a)
Since $v$ is $w$-bounded, $w \ge [w]_v v$ with $[w]_v > 0$.

Case 1: $\theta(w) = \theta ( [w]_v v)$.

Then $1 = [w]_v$ and $w \ge v $. So $rv \le  B(v) \le B(w) \le s w$.
We apply $\theta$ and obtain $r \le s$.

Case 2: $\theta(w) > \theta ( [w]_v v)$

Since $B$ is strictly $\theta$-increasing, there exists some $\delta > 0$
such that
\[
sw \ge B(w) \ge (1+\epsilon) B ( [w]_v v) = (1+\epsilon) [w]_v B (  v)
\ge (1+\epsilon) [w]_v r  v,
\]
which implies that $[w]_v \ge (1+\epsilon) (r/s) [w]_v$
and so $ r < s$.

In either case $r \le s$. If $r =s$, the second case cannot occur
and the first case holds where $w \ge v$.

(b) From part (a), by symmetry, $r=s$ and $w \ge v$. Again, by symmetry,
$w=v$.

If just $\theta(v) >0$ and $\theta(w) >0$, we
set $\tilde v = \frac{1}{\theta(v)} v $
and $\tilde w = \frac{1}{\theta(w)} w $. Then $\theta(\tilde v) =1 =
\theta(\tilde w)$ and $B \tilde v \ge r \tilde v$ and $B\tilde w \le s \tilde w$.

We apply the previous considerations to $\tilde v$ and $\tilde w$
and obtain the results.
\end{proof}


\section{Order-bounded maps}
\label{sec:order-bounded}

The following terminology has been adapted from various works by Krasnosel'skii
\cite[Sec.2.1.1]{Kra} and
coworkers  \cite[Sec.9.4]{KrLiSo} though it has been modified.

\begin{definition}
\label{def:u-bounded}
Let $B: X_+ \to X_+$, $u \in X_+$. $B$ is called {\em pointwise $u$-bounded}
if, for any $x \in X_+$, there exist  some $n \in \N$
and $\gamma > 0$ such that $B^n (x) \le \gamma u$. The point $u$
is called a {\em pointwise order bound} of $B$.

$B$ is called {\em
uniformly $u$-bounded} if there exists
some $c > 0$ such that  $ B ( x )\le c \|x\| u$ for all
$x \in X_+$. The element $u$ is called a {\em uniform order bound} of $B$.

$B$ is called {\em uniformly order-bounded} if it is uniformly $u$-bounded
for some $u \in X_+$.
$B$ is called {\em pointwise order-bounded} if it is pointwise $u$-bounded for
some $u \in X_+$.

\end{definition}

If $B: X_+ \to X_+$ is bounded and $X_+$ is solid, then $B$
is uniformly $u$-bounded for every interior point $u$ of $X_+$.

Uniform order boundedness is preserved if the original norm $\|
\cdot \|$ is
replaced by its monotone companion norm $\sharp \cdot\sharp$.

\begin{proposition}
\label{re:order-bounded-compan}
Let $B: X_+ \to X_+$ be order-preserving.
Let $u \in X_+$ and $B$ be uniformly $u$-bounded. Then $B$ is also
uniformly $u$-bounded with respect to the monotone companion norm.
\end{proposition}

\begin{proof}
Let $x \in X_+$. By (\ref{eq:mod-norm}), for each $n \in \N$ there
exists some $y_n \in X_+$ such that $x \le y_n$ and
$\sharp x\sharp \le \|y_n\| \le \sharp x\sharp + (1/n)$. Since $B$ is
order-preserving and uniformly $u$-bounded,
\[
B(x) \le B(y_n) \le c \|y_n\| u , \qquad n \in \N.
\]
We take the limit as $n \to \infty$ and obtain $B(x) \le c \sharp x\sharp u$.
\end{proof}

\begin{proposition}
\label{re:u-bounded-pointw-unif}
 Let $X_+$ be a  complete
cone, $u \in X_+$,   and $B: X_+ \to X_+$ be continuous, order-preserving and  homogeneous.
Then the following hold.

\begin{itemize}
 \item[(a)] $B$ is uniformly $u$-bounded if for any $x\in X_+$ there exists some $c =c_x \ge 0$
 such that $B (x) \le c u$.

\item[(b)] If $B$ be pointwise $u$-bounded,
then some power of $B$ is uniformly $u$-bounded.
\end{itemize}

\end{proposition}

\begin{proof}
We prove (b); the proof of (a) is similar.
Define
\[
M_{n,k} = \{ x \in X_+; B^n (x) \le k u \}, \qquad n,k \in \N.
\]
Since $B$ is continuous and $X_+$ is closed,
 each set $M_{n,k}$ is closed. Since $B$ is assumed to be pointwise
 $u$-bounded, $X_+ = \bigcup_{k,n\in \N}
M_{n,k}$. Since $X_+$ is a complete metric space, by the Baire
category theorem, there exists some $n,k \in \N$ such that
$M_{n,k}$ contains a relatively open subset of $X_+$: There
exists some $y \in X_+$ and $\epsilon >0$ such that
$y + \epsilon z \in M_{n,k}$ whenever $z \in X$, $\|z\|\le 1$,
and $y + \epsilon z \in X_+$. Now let $z \in X_+$ and $\|z\|\le 1$.
Since $B$ is order-preserving and $y +\epsilon z \in X_+$, $ B^n (\epsilon z) \le B^n( y +  \epsilon z) \le k u$.
Since $B$ is  homogeneous, for all $x\in \dot X_+$,
\[
B^n (x) = \frac{\|x\|}{\epsilon}  B^n \Big (\frac{\epsilon}{\|x\|} x \Big) \le
\frac{k}{\epsilon} \|x\| u. \qedhere
\]
\end{proof}


\section{Upper semicontinuity of the companion spectral radius}
\label{sec:semicon}
In order to be able to compare the companion spectral radius to
the upper Collatz-Wielandt bound which will be defined later, some results on the
upper semicontinuity of the cone spectral radius are useful.
For more results of that kind see \cite{LeNu13}.

In the following, let $X$ be an ordered normed vector space with cone
$X_+$.

\begin{lemma}
\label{re:conv-power}
Let $B$ be  bounded  and homogeneous, $x \in X_+$, and $B$ be continuous
at $B^n(x)$ for all $n \in \N$.
Let $(B_k)$ be a sequence of
bounded homogeneous maps such that $\|B_k - B\|_+ \to 0$
as $k \to \infty$ and $(x_k)$ be a sequence in $X_+$ such that
$x_k \to x$. Then, for all $n \in \N$,  $B_k^n (x_k) \to B^n(x)$ as $k \to \infty$.
\end{lemma}

\begin{proof} For $k \in \N$,
\[
\begin{split}
\|B_k (x_k ) - B (x) \| & \le \|B_k (x_k) - B (x_k) \| + \|B(x_k) - B (x) \|
\\
\le &
\|B_k - B \|_+\; \|x_k\| + \|B(x_k) - B (x) \| \stackrel{k\to \infty}{\longrightarrow} 0. \qedhere
\end{split}
\]
This provides the basis step for an induction proof. The induction
step follows in the same way.
\end{proof}

\begin{theorem}
\label{re:specrad-semicon2}
 Let $u \in X_+$, $u \ne 0$. Let $B: X_+ \to X_+$
be homogeneous and bounded and $B$ be continuous at $B^n(u)$
for all $n \in \N$. Let $(B_k)$ be a sequence of
bounded, homogeneous, order preserving maps such that $\|B_k - B\|_+ \to 0$
as $k \to \infty$. Assume that there exist $m \in \N$ and $c \ge 0$
such that $B_k^m(x) \le c\|x\| u$ for all $k \in \N$ and all $x \in X_+$.
Then
\[
\limsup_{k\to \infty} \br_+^\sharp (B_k) \le \gamma_B^\sharp(u) \le \br_o^\sharp(B)\le \br_+^\sharp(B).
\]
\end{theorem}

\begin{proof}
We can also assume that $\|u\|=1$.
Choose $m\in \N$ and $c \ge 0$ as in the statement of the theorem.
By Proposition \ref{re:order-bounded-compan} and its proof,
\[
B^m_k (x) \le c \sharp x\sharp u , \qquad k \in \N, x \in X_+.
\]
Since $B_k$ is order-preserving and homogeneous,
\[
B_k^{n+m} (x) \le c \sharp x\sharp B_k^n (u) , \qquad k,n\in \N, x \in X_+.
\]
We apply the monotone companion norm,
\[
\sharp B_k^{n+m} x \sharp \le   \sharp x\sharp \; \sharp B_k^n (u)\sharp, \qquad k,n \in \N, x \in X_+ .
\]
So
\[
\sharp B_k^{n+m}\sharp_+ \le  \sharp B^n_k (u)\sharp, \qquad n \in \N.
\]
Let $r > \gamma_B^\sharp(u)$. Then there exists some $N \in \N$ such that
\[
\sharp B^n (u) \sharp < r^n, \qquad n > N.
\]
Let $n \in \N$, $n > N$.
Since $B_k^n (u) \stackrel{k \to \infty}{\longrightarrow} B^n (u)$ by Lemma \ref{re:conv-power}, there exists some $k_n \in \N$ such that
\[
\sharp B_k^n (u) \sharp < r^n, \qquad k \ge k_n.
\]
We combine the inequalities.
\[
\sharp B_k^{n+m}\sharp_+ \le  r^n, \qquad k \ge k_n.
\]
Then
\[
\br_+^\sharp(B_k) \le \sharp B_k^{n+m}\sharp_+^{1/(n+m)}  \le  r^{n/(n+m)}, \qquad k \ge k_n.
\]
So, for any $n > N$,
\[
\limsup_{k \to \infty}\br_+^\sharp(B_k) \le  r^{n/(n+m)} .
\]
We take the limit as $n \to \infty$ and obtain
\[
\limsup_{k \to \infty}\br_+(B_k) \le r.
\]
Since this holds for any $r > \gamma_B^\sharp(u)$, the assertion follows.
\end{proof}



\section{Upper Collatz-Wielandt numbers}
\label{sec:CWupper}


The {\em upper Collatz-Wielandt number} of $B$ at $x \in \dot X_+$ is defined
as
\begin{equation}
\label{eq:upper-CW-num}
\|B\|_x = \|B (x) \|_x = \inf \{r \ge 0; B (x) \le r x\}.
\end{equation}
where $\|B\|_x = \infty$ if $B (x) $ is not $x $-bounded.

{\em Collatz-Wielandt numbers} \cite{FoNa}, without  this name,
became more widely known when Wielandt used them to give a new proof of the Perron-Frobenius
theorem \cite{Wie}. We will use them closer to Collatz' original purpose,
namely to prove  inclusion theorems
(Einschlie\ss ungss\"atze) for $\br_+(B)$
 which generalize those in \cite{Col1, Col2}.


\subsection{Upper local Collatz-Wielandt radius}

The {\em upper local Collatz-Wielandt  radius} of $x \in X_+$ is defined
as
\begin{equation}
\label{eq:spec-rad-Col-Wie-loc-up}
\eta^x(B) = \inf_{n\in \N} \|B^n\|_x^{1/n}.
\end{equation}

\begin{lemma}
\label{re:Col-upp-loc}
Let $B,C: X_+ \to X_+$ be homogeneous and order-preserving.
Let
$x \in \dot X_+$,  and $B(x)$ and $C(x)$ be $x$-bounded.
Then $C(B(x))$ is $x$-bounded and $\|CB\|_x \le \|C\|_x \|B\|_x$.
\end{lemma}

\begin{proof}
By (\ref{eq:order-max-for}),
$C (x) \le \|C\|_x x$ and $B(x) \le \|B\|_x x$ and
$CB (x) \le \|B\|_x C (x) \le \|B\|_x \|C\|_x x$; so
$\|CB\|_x \le \|B\|_x \|C\|_x$.
\end{proof}

\begin{lemma}
\label{re:Col-Wie-upp-loc}
Let $u \in \dot X_+$.

\begin{itemize}

\item[(a)] Then $\eta^u(B^m) \ge (\eta^u(B))^m$ for all $m \in \N$.

\item[(b)] If there exists some $k \in \N$ such that $B^m u $ is $u$-bounded for all $m \ge k$,
then
$\eta^u(B) =  \lim_{n\to \infty} \|B^n\|_u^{1/n}< \infty$
and $\eta^u(B^m) = (\eta^u(B) )^m$ for all $m \in \N$.
\end{itemize}
\end{lemma}

\begin{proof} (a) Let $u \in \dot X_+$, $m \in \N$. Then
\[
\eta^u(B^m) = \inf_{n\in \N} \|B^{mn}\|_u^{1/n}=
(\inf_{n\in \N} \|B^{mn}\|_u^{1/(mn)})^m
\ge
(\inf_{k\in \N} \|B^{k}\|_u^{1/k})^m.
\]

(b) Now let $k \in \N$ such that $B^m(u)$ is $u$-bounded for all $m \ge k$.
By Lemma \ref{re:Col-upp-loc},
\[
c_{n+m} \le c_n c_m, \qquad n, m \ge k, \qquad c_n = \|B^n\|_u.
\]
Let $r $ be an arbitrary number such that
$\eta^u(B) = \inf_{n\in \N} \|B^n \|_u^{1/n} < r$.
Then there exists some $m \in \N$
 such that
$c_m^{1/m}= \|B^m\|_u^{1/m} \le  r$.
So $ B^m (u) \le r^m u$. By applying $B^m$ as often as necessary,
we can assume that $m \ge k$.

Any number $n \in \N$ with $n \ge 2m$
has a unique representation $n = pm + q $ with $p \in \N$ and  $m \le q < 2m$.
Then, for $n \ge 2m$,
\[
c_n \le  c_m^p c_q \le r^{mp} c_q.
\]
If $c_q=0$, then both the limit inferior and the limit are zero and equal.
So we can assume that $c_q \ne 0$. We have
\[
c_n^{1/n} \le r^{pm/n} c_q^{1/n}.
\]
As $n \to \infty$, $pm/n \to 1 $ and $\limsup_{n\to \infty} c_n^{1/n} \le r$.
Since $r$ was any number larger than the infimum , the limes
superior and inferior coincide and the limit exists and equals the infimum.

The second equality in (b) follows from the fact that every subsequence of a convergent
sequence converges to the same limit.
\end{proof}

\begin{remark} If $B(u)$ is $u$-bounded, $\|B\|_u$ is the cone norm
of $B$ in the space $X_u$ with $u$-norm. Then $\eta^u(B)$ is the
asymptotic growth factor of $u$ taken in $X_u$.
\end{remark}

\begin{theorem}
\label{re:growth-factor-upper-est-power}
Let $B: X_+ \to X_+$ be homogeneous, bounded, and order-preserving.
Let $u \in X_+$, $\alpha \in \R_+$ and $k \in \N$
such that $ B^k (u) \le \alpha^k u$.

Let $u$ be a normal point of $X_+$ or   $B$ be power-compact. Then $\gamma_B(u) \le \alpha $.
\end{theorem}

\begin{proof}
Since $\gamma(u,B^k) = (\gamma(u,B))^k$ by  (\ref{eq:geom-fac-power}), we can assume that $k=1$.
Since $B$ is homogenous, it is enough to show that $B(u) \le u$
implies that $\gamma_B(u) = \gamma(u,B) \le 1$.

Let $B(u) \le u$. Then $B^n(u) \le u$ for all $n \in \N$.

We first assume that $u$ is a normal point of $X_+$.
By Definition \ref{def:normal-point},
there exists some $c > 0$ such that $\|B^n(u)\| \le c $
for all $n \in \N$. This implies $\gamma_B(u) \le 1$.

Now assume that $B^\ell$ is compact for some $\ell \in \N$ and
that $\gamma_B(u) > 1$. Then the
sequence $(a_n)$ with
\begin{equation}
a_n =\|B^n (u)\|
\end{equation}
is unbounded.
 By a lemma by Bonsall \cite{Bon58},
 there exists a subsequence $a_{n_j}$ such that
\[
a_{n_j} \to \infty, j \to \infty, \qquad a_k \le a_{n_j}, \qquad k=1, \ldots, n_j,
\quad j \in \N.
\]
Set $v_j = \frac{1} {a_{n_j}  }B^{n_j} (u) $. Then
\[
v_j = B^\ell (w_j), \qquad w_j = \frac{1} {\|B^{n_j} (u) \|}B^{n_j-\ell} (u).
\]
Now
\[
\|w_j\| \le \frac{a_{n_j-\ell} } {a_{n_j} }
\le 1.
\]
So, after choosing a subsequence, $(v_j)$ converges to some $v \in X_+$,
$\|v\|=1$. Since $B^n(u) \le u$ for all $n \in \N$,
\[
v_j \le \frac{1}{a_{n_j}} u.
\]
Since $a_{n_j } \to \infty$ and $X_+$ is closed, we have $v \le 0$, a
contradiction.
\end{proof}

\begin{corollary}
\label{re:growth-factor-upper-est-Coll}
Let $B: X_+ \to X_+$ be homogeneous, bounded, and order-preserving.
Let $u \in \dot X_+$ and $B^k(u)$ be $u$-bounded for all but finitely
many $k \in \N$. Assume that $u$ is a normal point of $X_+$ or that
$B$ is power-compact. Then $\gamma_B(u) \le \eta^u(B)$.
\end{corollary}

\begin{proof}
By definition, $B^k (u) \le \|B^k\|_u u$ for all
$k \in \N$ with $\|B^k\|_u < \infty$. By Theorem \ref{re:growth-factor-upper-est-power}, $\gamma_B(u) \le \|B^k\|_u$
for all $k \in \N$. So $\gamma_B(u) \le \eta^u(B)$.
\end{proof}

\begin{theorem}
\label{re:orb-rad-upper-est-power}
Let $B: X_+ \to X_+$ be homogeneous, bounded,  and order-preserving.
Let $u \in X_+$, $\alpha \in \R_+$ and $k \in \N$
such that $ B^k (u) \le \alpha^k u$.

Let $u$ be a normal point of $X_+$
 Then $\br_o(B) \le \alpha $ if $B$ is pointwise $u$-bounded,
and  $\br_+(B) \le \alpha $ if some power of $B$ is uniformly $u$-bounded,

\end{theorem}

\begin{proof} As in the proof of Theorem \ref{re:growth-factor-upper-est-power},
we can reduce the proof to the implication
\[
B(u) \le u \implies \br_o(B) \le 1.
\]
Assume that $B(u) \le u$. Let $x \in X_+$. Since $B$ is pointwise
$u$-bounded, there exists some $m \in \N$ and $c > 0$ such that
$B^m(x) \le c u$. For all $n \in \N$, $B^{m+n} (x) \le c B^n(u) \le c u$.
Since $u$ is a normal point of $X_+$, by Definition \ref{def:normal-point},
there exists some $\tilde c > 0$ such that $\|B^{m+n}(x) \| \le c \tilde c$
for all $n \in \N$. This implies $\gamma_B(x) \le $1. Since $x \in X_+$
has been arbitrary, $\br_o(B) \le 1$.

If $B^m$  is uniformly $u$-bounded, we can replace $c$ by $c \|x\|$
and we obtain $\br_+(B) \le 1$.
\end{proof}

Similarly as for Corollary \ref{re:growth-factor-upper-est-Coll},
this yields the following result.

\begin{corollary}
\label{re:orb-rad-ColWie}
Let $B: X_+ \to X_+$ be homogeneous, bounded, and order-preserving.

Let $u$ be a normal point of $X_+$.
 Then $\br_o(B) \le \eta^u(B) $ if
 $B$ is pointwise $u$-bounded, and $\br_+(B) \le \eta^u(B) $ if
some power of  $B$ is uniformly $u$-bounded.
\end{corollary}

For those $x\in X_+$ for which the sequence $\sharp B^n x \sharp^{1/n} $ is bounded,
we extend the definition of the {\em companion growth bound} of the $B$-orbit of $x$ by
\begin{equation}
\gamma_B^\sharp (x) := \limsup_{n\to \infty} \sharp B^n (x) \sharp^{1/n}
\end{equation}
and set it equal to infinity otherwise. We extend the definition of
 the {\em orbital companion spectral radius} of $B$ by
\begin{equation}
 \br_o^\sharp (B):= \sup_{x \in X_+} \gamma_B^\sharp(x).
\end{equation}

\begin{theorem}
\label{re:Col-Wie-companion}
Let $B: X_+ \to X_+$ be homogeneous, bounded,  and order-preserving, $u \in X_+$.
Let $B$ be pointwise $u$-bounded and assume that there is some $\ell \in \N$
such that $B^n( u)$ is $u$-bounded for all $n \ge \ell$.

Then
\[
\eta_x(B) \le \gamma_B^\sharp (x) \le \gamma_B^\sharp(u) \le \eta^u(B)
\]
for all $x \in X_+$ and
\[
cw(B) \le \br_{cw} (B) \le \br_o^\sharp (B)= \gamma_B^\sharp(u) \le \eta^u(B).
\]

If $B^m$ has the lower KR property for some $m \in \N$, $\gamma_B(u)\le
\br_+(B) \le \eta^u(B)$.
\end{theorem}

\begin{proof}
Let $x \in X_+$. We can assume $x \ne 0$.
Since $B$ is pointwise $u$-bounded, there exists some $k = k(x)\in \N$ and
some $c = c(x) > 0$ such that
$ B^k(x) \le c u $. Since $B$ is order-preserving and homogeneous, $B^{n} (x) \le c
B^{n-k} (u)$ for all $n >k$.
This implies
\[
   B^{n} (x) \le  c B^{n-k} (u) \le c \|B^{n-k}\|_u u.
\]
We apply the monotone companion norm,
\[
\sharp B^{n} (x) \sharp \le  c \sharp B^{n-k} (u) \sharp
 \le c \|B^{n-k}\|_u \sharp u \sharp.
\]
So
\[
 \sharp B^{n} (x) \sharp^{1/n} \le
 c^{1/n} \sharp B^{n-k} (u) \sharp^{1/n}
 \le \|B^{n-k}\|_u^{\;1/n} \big ( c \sharp u \sharp\big )^{1/n}.
\]
We take the limit as $n \to \infty$,
use Lemma \ref{re:Col-Wie-upp-loc} (b),
recall Proposition
\ref{re:Col-Wie-lower} and obtain
the first inequality. The second then
follows by taking the supremum over $x \in X_+$
and recalling $cw(B) \le {\bf r}_{cw}(B)$
from Proposition \ref{re:Col-Wie-lower}.

Assume that $B$ is bounded and $B^m$ has the lower KR property for some $m \in \N$.
We can assume that $r = \br_+(B) > 0$.
Then $B^m(v) \ge r^m v$
with $r = \br_+(B)$ and some $v \in \dot X_+$. Since
$B$ is pointwise $u$-bounded, there exists some $c > 0$ and $k \in \N$
(which depend on $v$)
such that $B^k (v) \le c u$. So $r^{m+k} v = B^{m+k}(v) \le c B^m(u)$.
For all $n \in \N$, $r^{m+k +n }( v )\le c B^{m+n}( u)$. Then
\[
r^{k+j} \|v\|_u \le c \| B^{j}  \|_u, \qquad j \in \N, j \ge m.
\]
So
\[
r \le (c/ r^k \|v\|_u)^{1/j} \| B^{j}  \|_u^{1/j}, \qquad j \in \N, j \ge m.
\]
Taking the limit as $j \to \infty$ yields the desired result.
\end{proof}

\begin{theorem}
\label{re:Col-Wie-companion2}
Let $B: X_+ \to X_+$ be homogeneous, bounded and order-preserving, $u \in X_+$.
Let some power of $B$ be uniformly $u$-bounded.

Then
\[
cw(B) \le \br_{cw} (B) \le \br_o^\sharp (B)= \br_+^\sharp (B)= \gamma_B^\sharp(u) = \eta^u(B) \le \gamma_B (u) \le \br_+(B).
 \]
Under additional assumptions, the following hold:

\begin{itemize}

\item If $u$ is a normal point in $X_+$, then
$\displaystyle \eta^u(B) = \br_o(B)= \br_+(B)   =
\gamma_B(u) = \lim_{n\to \infty} \|B^n(u)\|^{1/n}$.

\item  If some power of $B$ has the lower KR property, then
$\br_{cw} (B) = \eta^u(B) = \gamma_B(u) = \br_o(B)= \br_+(B)$.

\item If $B$ has the lower KR property, then
$cw(B) = \br_{cw} (B) = \eta^u(B) = \gamma_B(u) = \br_o(B)= \br_+(B)$.

\end{itemize}
\end{theorem}

\begin{proof}
Let $k \in \N$ such that $B^k$ is uniformly $u$-bounded.
By Proposition \ref{re:order-bounded-compan},
 $B^k$ is also uniformly $u$-bounded with respect to the
monotone companion norm.
Then there exists some $c > 0$ such that $B^k(x) \le c \sharp x \sharp u$
for all $x \in X_+$. For all $n \in \N$, since $B$ is order-preserving
and homogeneous,
\[
\begin{split}
B^{k+n} (x) = & B^k (B^n(x)) \le \sharp B^n(x)  \sharp c u,
\\
B^{k+n} (x)= & B^n (B^k(x)) \le c \sharp x \sharp B^n (u).
\end{split}
\]
By definition of upper Collatz-Wielandt numbers,
\[
\|B^{k+n}  \|_u \le \sharp B^n(u) \sharp c, \qquad n \in \N.
\]
By (\ref{eq:spec-rad-Col-Wie-loc-up}),
\[
(\eta^u(B))^{(k+n)/n} \le \sharp B^n(u) \sharp^{1/n} c^{1/n}, \qquad n \in \N.
\]
We take the limit as $n \to \infty$,
\begin{equation}
\label{eq:upperCWrad-growth}
\eta^u(B) \le \liminf_{n\to \infty} \sharp B^n(u) \sharp^{1/n}
\le
\liminf_{n\to \infty} \| B^n(u) \|^{1/n}.
\end{equation}
This implies $ \eta^u(B)\le \gamma_B^\sharp(u) $.
The other inequality follows from Theorem \ref{re:Col-Wie-companion}.

Let $u$ be a normal point. By Corollary \ref{re:orb-rad-ColWie},
\[
\eta^u (B) \ge \br_+(B) \ge \br_o(B) \ge \gamma_B(u) =
\limsup_{n\to \infty} \|B^n(u)\|^{1/n}.
\]
Together with (\ref{eq:upperCWrad-growth}), this implies equalities.

Since the companion norm is order-preserving,
\[
\sharp B^{k+n}(x) \sharp \le c \sharp x \sharp \; \sharp B^n (u) \sharp
\]
and
\[
\sharp B^{k+n} \sharp \le c  \sharp B^n (u) \sharp, \quad n \in \N.
\]
Since $B$ is bounded, $\br_+^\sharp (B) \le \gamma^\sharp_B(u)$.

The other statements now follow from the previous theorems.
\end{proof}

In view of estimating the cone spectral radius from above the following observation may be of interest.

\begin{corollary}
\label{re:radius-estimate-above}
Let $B: X_+ \to X_+$ be a  homogeneous, bounded, order-preserving map.
Assume that $X_+$ is normal and complete or  some power of $B$ has
the lower KR property.

Then $\br_+(B)$ is a lower bound for all upper Collatz-Wielandt numbers
$\|B\|_u$ where $u \in \dot X_+$, $B(u)$ is $u$-bounded and $B$ is
pointwise $u$-bounded.
\end{corollary}

\begin{proof} Combine the previous theorems with (\ref{eq:spec-rad-Col-Wie-loc-up})
and recall that $\br_o^\sharp (B) = \br_+(B)$ if $X_+$ is normal and
complete.
\end{proof}


\subsection{The upper Collatz-Wielandt bound}


Let $u \in \dot B_+$ and $B(u)$ be  $u$-bounded. Then $B(x) $ is $x$-bounded
for any $u$-comparable $x \in X_+$.

So we define the upper Collatz-Wielandt
bound with respect to $u$ by
\begin{equation}
CW_{u}(B) = \inf \{ \|B\|_x; x \in X_+, x \sim u \}.
\end{equation}

If $x \in X_+$ and $x \sim u$, $\eta^x(B) = \eta^u (B)$. Since $\eta^x(B) \le
\|B\|_x$,
\begin{equation}
\label{eq:upperCWbound}
\eta^u(B) \le CW_u(B).
\end{equation}

We have the following inequalities from
Theorem \ref{re:Col-Wie-companion} and Theorem \ref{re:Col-Wie-companion2}.

\begin{theorem}
\label{re:Coll-upper-bound-gen}
Let $B$ be homogeneous and order-preserving.
Let $u \in \dot X_+$, $B(u) $ be $u$-bounded and $B$ be
pointwise $u$-bounded.
Then
\[
cw(B) \le \br_{cw}(B) \le  \eta^u(B) \le CW_u(B).
\]
\end{theorem}

Lower KR property of the map turns some of the inequalities
in equalities (Theorem \ref{re:Col-Wie-companion2}).

\begin{theorem}
\label{re:Coll-upper-bound}
Let $B$ be homogeneous, bounded, and order-preserving
and some power of $B$ have the lower KR property. Let $u \in \dot X_+$
and some power of $B$ be uniformly $u$-bounded. Then
\[
\br_{cw}(B) = \br_o(B) = \br_+(B) = \eta^u(B) \le CW_u(B).
\]
\end{theorem}


\subsection{Monotonicity of the  spectral radii and the Collatz-Wielandt radius}

If the cone $X_+$ is normal, the cone and orbital spectral radius
are increasing functions of the homogeneous bounded order-preserving maps
(cf. \cite[L.6.5]{AGN}). Collatz-Wielandt numbers, bounds, and radii
and the companion radii are increasing functions of the map even
if the cone is not normal.

\begin{theorem}
\label{re:spec-rad-increasing}
Let $A,B: X_+ \to X_+$ be bounded and  homogeneous.
Assume that $ A(x) \le B (x) $ for all $x\in X_+$ and that $A$ or $B$
are order-preserving.

Then $cw(A) \le cw(B)$,
 $\br_{cw}(A) \le \br_{cw} (B)$, $\br_+^\sharp (A) \le \br_+^\sharp (B)$
 $\br_o^\sharp (A) \le \br_o^\sharp (B)$. Further, for all $x \in X_+$,
 $\|A\|_x \le \|B\|_x$ and $\eta^x(A) \le \eta^x (B)$.

If $u \in \dot X_+$ and $A(u)$ and $B(u)$ are $u$-bounded,
then $CW_u(A) \le CW_u(B)$.

If $X_+$ is a normal cone, then also $\br_+(A) \le \br_+ (B)$ and $\br_o(A) \le \br_o(B)$.
\end{theorem}

\begin{proof}  We claim that $A^n (x) \le B^n (x)$ for all $x \in X_+$
and all $n \in \N$. For $n =1$, this holds by assumption.
Now let $n \in \N$ and assume the statement holds for $n$.
If $A$ is order-preserving,
then, for all $x \in X_+$, since $B^n (x) \in X_+$,
\[
A^{n+1} (x) = A (A^n (x)) \le A (B^n (x))  \le B (B^n (x)) = B^{n+1} (x).
\]
If $B$ is order-preserving,
then, for all $x \in X_+$, since $A^n (x) \in X_+$,
\[
A^{n+1} (x) = A (A^n (x)) \le B (A^n (x))  \le B (B^n (x)) = B^{n+1} (x).
\]
Since the companion norm is order-preserving,
$\sharp A^n (x) \sharp \le  \sharp B^n (x) \sharp $ for all $x \in X_+$, $n \in \N$.
Further $\sharp A^n \sharp_+ \le  \sharp B^n \sharp_+$ for all $n \in \N$ and
$\br_+^\sharp(A) \le \br_+^\sharp(B)$.
Further $\gamma_A^\sharp(x) \le \gamma_B^\sharp(x)$
and so $\br_o^\sharp(A) \le \br_o^\sharp(B)$.

If $X_+$ is normal, the respective order radii taken with the original norm
coincide with those taken with the companion norm.

As for the Collatz-Wielandt radius,
\[
B^n x \ge A^n x \ge [A^n]_x x.
\]
By (\ref{eq:Col-Wie-num-low}), $[B^n ]_x \ge [A^n]_x$, and the claim follows from
(\ref{eq:spec-rad-Col-Wie}) and (\ref{eq:spec-rad-Col-Wie-loc}).
The proofs for $\|\cdot\|_x$, $\eta^x$, and $CW_u$ are similar.
\end{proof}

\begin{proof}[Proof of Theorem \ref{re:spec-rad-increasing2}]
By Theorem \ref{re:spec-rad-increasing}, ${cw} (A) \le {cw}(B)$.
The assertion now follows from Theorem \ref{re:almost-all-equal}.
\end{proof}


\subsection{The upper Collatz-Wielandt bound as eigenvalue}


Conditions which make $CW_u(B)$  an eigenvalue of $B$ with
positive eigenvector and  imply equality
between all these numbers including $CW_u(B)$ can be found in
\cite[Thm.7.3]{AGN}.
Using  the companion half-norm $\psi$, one can drop that the cone is normal
and complete provided that the map is compact. Solidity of the
cone can be replaced by the weaker assumption that the map
is uniformly $u$-bounded.

\begin{theorem}
\label{re:upperCW-eigen}
Let $B:X_+ \to X_+$ be continuous, compact, homogeneous, and order-preserving.
Let $u \in \dot X_+$ and $B$ be uniformly $u$-bounded.

Then $cw(B) = \br_{cw}(B) = \br_+(B) = \eta^u(B) = CW_u(B)$.

If $r = CW_u(B) > 0$, then there exists some $v \in \dot X_u$ such that
$B(v) =r v$.
\end{theorem}

\begin{remark}
If $X_+$ is complete, we also obtain this result
if we replace compactness of $B$ by the assumptions in Theorem
\ref{re:eigen-noncomp-special} with part (a) or by assumption (ii)
in \cite[Thm.7.3]{AGN}.
\end{remark}

More generally, the following holds.

\begin{theorem}
\label{re:robust}
Let $B: X_+ \to X_+$ be homogeneous and  order-preserving.
Let $u \in \dot X_+$ and $B$ be uniformly $u$-bounded
 and  continuous
at $B^n(u)$ for all $n \in \N$.
Assume there is some $\epsilon_0 > 0$ such that,
for all   $\epsilon \in (0, \epsilon_0)$,
 the maps $B_\epsilon$,
$B_\epsilon (x) = B(x) + \epsilon \psi (x) u$,
have eigenvectors $B_\epsilon (v_\epsilon) = \lambda_\epsilon v_\epsilon$
with $v_\epsilon \in \dot X_+$ and $\lambda_\epsilon > 0$.

Then $ \br_+(B) \ge \gamma_u(B) \ge  CW_u(B) = \eta^u(B)$ with equality holding
everywhere if $u$ is a normal point of $X_+$ or
some power of $B$ has the lower KR property.

Further, if $B$ has the KR property and $CW_u(B) > 0$, there
exists some $v \in \dot X_+$ such that $B (v) = CW_u(B) v$.

\end{theorem}

\begin{proof}
Choose a sequence $(\epsilon_n)$ in $(0,\epsilon_0)$ with
$\epsilon_n \to 0$.  Set $B_n = B_{\epsilon_n}$. The maps
$B_n$ inherit uniform $u$-boundedness from $B$.

By assumption, there exist $v_n \in \dot X_+$ and $r_n > 0$ such
that $B (v_n) + \epsilon_n \psi( v_n) u =r_n v_n$.
Since $B$ is uniformly $u$-bounded and $\psi(v_n) >0$, $v_n$ is
$u$-comparable. By (\ref{eq:upperCWbound}), $ r_n \ge CW_u(B_n)$.
Also $r_n \le cw(B_n)$ by (\ref{eq:low-CW-bound}).
By Theorem \ref{re:Coll-upper-bound-gen} and Theorem \ref{re:Col-Wie-companion2},
$\eta^u(B_n) = \br_+^\sharp (B_n) = CW_u(B_n)$ for all $n \in \N$.
Further, $CW_u(B_n) \ge CW_u(B)$.

Suppose that $\br_+^\sharp (B)<CW_u(B)$.
Since $\epsilon_n \to 0$, $\|B_n - B \|_+ \to 0$.
By Theorem \ref{re:specrad-semicon2}, $\br_+^\sharp(B_n) < CW_u(B)$
for large $n$, a contradiction.

So $\br_+^\sharp (B) \ge CW_u(B)$.
By Theorem \ref{re:Col-Wie-companion2}, also $\eta^u(B) = \br_+^\sharp (B)$.
Since $\eta^u(B) \le CW_u(B)$, we have $\eta^u(B) = CW_u(B)$.
The other inequalities follow from Theorem \ref{re:Col-Wie-companion2}.

If some power of $B$ has the lower KR property,
equality holds by Theorem \ref{re:Coll-upper-bound}.

If $u$ is a normal point of $X_+$, equality follows from
Theorem \ref{re:Col-Wie-companion2}.

Assume that $B$ has the KR property and $CW_u(B)> 0$. Then $\br_+(B) =CW_u(B)> 0$
and there exists some $v \in \dot X_+$ such that $B(v) = \br_+(B) v$.
\end{proof}

The equality $\br_+(B) = CW_u(B)$
guarantees that, at least in theory,
 one can get arbitrarily sharp estimates of $\br_+(B)$
from above in terms of upper Collatz-Wielandt numbers $\|B\|_x$
by choosing an appropriate $x \in \dot X_+$ for which $B(x)$ is $x$-bounded
and $B$ pointwise
$x$-bounded (Corollary \ref{re:radius-estimate-above}).
Crude attempts in this direction are made for the rank-structured
discrete population model with mating in Section  \ref{sec:model}.

The idea of perturbing the map $B$ as above or in a similar way
is quite old; see \cite[Satz 3.1]{Sch55} and \cite[Thm.3.6]{Tho64}.

\begin{theorem}
\label{re:upperCW-eigen2}
Let $X_+$ be complete.
Assume that  $B = K + A$ where $K:X_+ \to X_+$ is compact, homogeneous, continuous
and order preserving and $A:X \to X$ is linear, positive and bounded
and $\br(A) < \br_+(B)$.
Let $u \in \dot X_+$ and $B$ be uniformly $u$-bounded.

Then $cw(B)= \br_{cw}(B) = \br_+(B) = \eta^u(B) = CW_u(B)$.

If $r = CW_u(B) > 0$, then there exists some $v \in \dot X_u$ such that
$B(v) =r v$.
\end{theorem}

\begin{proof}
For $\epsilon \in [0,1]$, we define $B_\epsilon :X_+ \to X_+$
by $B_\epsilon (x) = B (x) + \epsilon \psi (x)u$ where $\psi$
is the companion half-norm. Then $B_\epsilon = K_\epsilon + A$
with $K_\epsilon(x) = K (x) + \epsilon \psi(x) u $.
Then $K_\epsilon $ is compact, continuous, order-preserving
and homogeneous.

Since $A$ is linear and bounded, $B_\epsilon^n = K_{n,\epsilon} + A^n$ with
compact, continuous, homogeneous, order-preserving maps $K_{n,\epsilon}$.

If $n$ is chosen large enough, $\|A^n\| < \br_+(B^n)$.
By Theorem \ref{re:eigen-noncomp-special} (a), some power of $B$ has the KR property.
Since $B(x) \le B_\epsilon (x)$ for all $x \in X_+$, $\br_+(B)= cw(B)\le
cw(B_\epsilon) \le \br_+(B_\epsilon)$ by Theorem \ref{re:spec-rad-increasing2}.

So $\br(A) < \br_+(B) \le \br_+(B_\epsilon)$.
By Theorem \ref{re:eigen-noncomp-special} (a), for large enough $n$,  there exist
eigenvectors $v_\epsilon \in \dot X_+$ such that $B_\epsilon^n (v_\epsilon)
= r_\epsilon^n v_\epsilon $ with $r_\epsilon = \br_+(B_\epsilon)$.

It is easy to see that, for $\epsilon \in (0,1]$, $B_\epsilon $
is strictly $\psi$-increasing and that $B_\epsilon^n$ inherits this
property. Set $w_\epsilon= B_\epsilon (v)$. Then $B_\epsilon^n (w_\epsilon)
= r_\epsilon^n w_\epsilon $. Since $B_\epsilon$ is uniformly $u$-bounded
and $\psi(v_\epsilon) >0$, $w_\epsilon$ is $u$-comparable.
Since $B_\epsilon^k(v_\epsilon) \in \dot X_+$ for all $k \in \N$,
$B^n_\epsilon (v_\epsilon)$ is $u$-comparable and so $v_\epsilon $
and $w_\epsilon$ are comparable.

By Lemma \ref{re:Perron},
 $B_\epsilon (v_\epsilon ) = \alpha_\epsilon v_\epsilon$ for some $\alpha_\epsilon >0$
 which must then equal $r_\epsilon$.

Since $v_\epsilon$ is $u$-comparable $r_\epsilon= cw(B_\epsilon) = \br_{cw}(B_\epsilon)
= \br_+(B_\epsilon) = CW_u(B_\epsilon)$ for all $\epsilon \in (0,1]$.

By Theorem \ref{re:robust}, $CW_u(B) = \br_+(B) = \eta^u(B) = \gamma^u(B)$.

Now choose a decreasing sequence $(\epsilon_k)$ in $(0,1]$ with $\epsilon_n
\to 0$. Then $r_k = \br_+(B_{\epsilon_k})$ form a decreasing
sequence with $r_k \ge CW_u(B)$. Suppose $CW_u(B) >0$.
Set $v_k = v_{\epsilon_k}$. We can assume that $\|v_k\|=1$
for all $k \in \N$. Then
\[
r_k v_k = K(v_k) + A(v_k) + \epsilon_k \psi(v_k) u .
\]
Let $r = \lim_{k \to \infty} r_k$. Since $K$ is compact,
\[
(r - A) v_k = (r-r_k)v_k + K(v_k) + \epsilon_k \psi(v_k) u
\]
converges as $k \to \infty$ after choosing a subsequence.
Since $r \ge  CW_u(B) =\br_+(B)>
br_+(A)$,  $(r-A)^{-1} =
\sum_{j=1}^\infty (1/r)^{j+1} A^k$ exists as a continuous additive
homogeneous map and acts as the inverse of $r -A$. This implies
that $v_k \to v $ for some $v \in X_+$, $\|v\|=1$. Since $B$ is continuous,
$rv = B(v)$ which implies that $r = \br_+(B)$. Then $r \le cw(B)$
and equality follows. If $CW_u(B) =0$, equality holds anyway.
\end{proof}


\section{Monotonically compact order-bounded\\ maps on semilattices}
\label{sec:mono-comp}
As before, let $X$ be an ordered normed vector space with cone $X_+$.

\begin{definition}
\label{def:mono-compact}
Let $u\in \dot X_+$. $B:X_+ \to X_+ \cap X_u$
is called {\em antitonically  $u$-compact}
if $(B(x_n))$  has a convergent subsequence for each decreasing sequence $(x_n)$ in $X_+$ for which
there is some $c >0$ such that $x_n \le c u$ for all $n \in \N$.

$B$ is called {\em monotonically  $u$-compact}
if $(B(x_n))$  has a convergent subsequence for each monotone sequence $(x_n)$ in $X_+$ for which
there is some $c >0$ such that $x_n \le c u$ for all $n \in \N$.

\end{definition}


If $X_+$ is regular (Section \ref{sec:cones}),
then every order-preserving homogeneous $B: X_+ \to X_+$ is
monotonically $u$-compact for any $u \in X_+$.

\begin{definition}
\label{def:right-contin}
Let $u \in \dot X_+$,  $B: X_+ \to X_+ $ and $B(u)$ be $u$-bounded. $B$
is called antitonically continuous at $x \in X_+$ if, for
for any decreasing sequence $(x_k)$ in $X_+$ with $x \le x_k \le c u$ for all $k \in \N$ (with some
$c > 0$ independent of $k$) and $\|x_k -  x\| \to 0$,  we have $\psi( B(x_k) - B(x)) \to 0$.

$B$ is called monotonically continuous at $x \in X_+$ with for every
monotone sequence $(x_k)$ in $X_u$ with $x_k \le c u$ for all $k \in \N$
(with some
$c > 0$ independent of $k$) and $\|x_k - x\| \to 0$,  we have $\sharp B(x_k) - B(x) \sharp \to 0$.
\end{definition}

Recall the definition of a minihedral cone in Section \ref{subsec-cone-expose}
( $x\land y = \inf\{x,y\}$ exists for all $x,y \in X_+$)
and the upper local Collatz-Wieland spectral radius at $u \in X_+$,
\[
\eta^u(B)= \inf_{n\in \N} \|B^n\|_u^{1/n}.
\]

\begin{theorem}
\label{re:minihed-subeigen}
 Let $X_+$ be a  minihedral  cone and $B: X_+ \to X_+$ be
 order-preserving and homogeneous. Let $u \in \dot X_+$, and $B(u)$
 be $u$-bounded.
Assume that $B$  is
antitonically $u$-compact and antitonically continuous. Finally assume that
$\eta^u(B) > 0$ and $\|B(y_n)\|_u \to 0$
 for any decreasing sequence $(y_n)$ in $X_u \cap X_+$ with
$\|y_n \| \to 0$.

Then there exists some $x \in \dot X_+$,
such that $B (x) \ge \eta^u (B) x$ and
$\eta^u(B)= \br_o^\sharp(B) = \br_{cw}(B)= cw(B)$.

\end{theorem}

The first part of the proof has been adapted from \cite[L.9.5]{KrLiSo}
 where $B$ is assumed
to be a linear operator on the ordered Banach space $X$ and the cone $X_+$
to be normal. Use of the monotone companion metric allows to drop
normality as assumption. However, without $u$ being a normal point of the cone, compactness of $B$
may not imply monotonic compactness.

\begin{proof}[Proof of Theorem \ref{re:minihed-subeigen}]
Let $u \in X_+$ such that $B(u)$ is $u$-bounded.
Since $B$ is  homogeneous, we can assume that $\eta^u(B) =1$.
Otherwise, we consider $\frac{1}{\eta^u(B)} B$.
We define
\begin{equation}
\label{eq:mon-comp}
x_0= u, \qquad x_k = y_k \land u ,  \quad y_k =  B (x_{k-1}) + 2^{-k} u,\quad k \in \N.
\end{equation}
Then $x_k \le u=x_0$ for all $n\in \N$. By induction, since $B$ is
order-preserving, $x_{k+1}\le x_{k}$ for all $k \in \N$.
We apply the convergence principle in Proposition \ref{re:monotone-converge}
with $S$ being the set of sequences $(B(v_n))$ with $(v_n)$ being
decreasing and $v_1 \le c u$ for some $c >0$. $S$ has the properties
requested in Proposition \ref{re:monotone-converge}.
Since $B$ is antitonically
$u$-compact, every sequence in $S$ has a convergent subsequence.
So every sequence in $S$ converges.

Since $(B(x_n)) \in S$,  there exists some $z \in X_+$ such that  $(B (x_k))$
 converges to $ z $ as $k \to \infty$
and $B(x_k) \ge z$ for all $k \in \N$. By (\ref{eq:mon-comp}),
 $ y_k \to   z $. Further
\[
z \le B (x_{k-1} ) \le y_k \le B(u) + 2^{-k}u.
\]
By (\ref{eq:mon-comp}),
\[
x_k =  y_k \land  u \ge z \land  u=:x.
\]
 Notice that
\[
y_k \land  u + z - y_k \le z \land  u=x.
\]
So
\begin{equation}
\label{eq:minihed-subeigen}
0 \le x_k -x  \le y_k -z, \qquad k \in \N.
\end{equation}
Further
\[
\begin{split}
\sharp B(x_k) - B(x) \sharp = & \psi( B(x_k) - B(x)) \le \psi (B (y_k- z +x) -
B(x))
\\
\le &
\| B (y_k- z +x) -
B(x)\|.
\end{split}
\]
Also $x \le y_k - z + x \le B(u) + u + x \le (c\|u\| +2) u $. Recall that  $( y_k -z + x) $ converges $ x $
with respect to the the original norm.

Since $B$ is antitonically continuous,
  ($B(x_k))$ converges  to $B(x)$ with respect to the
monotone companion norm. Since $B(x_k) \to z$ with respect to the monotone
companion norm, we have $B(x) =z \ge x$.

Moreover,
$
x = z \land  u  = B(x) \land  u.
$

It remains to show that $x \ne 0$. Suppose that $x =0$. Then $z = B(x) =0$.
Recall that  $x_k =  y_k \land  u $ and $\|y_k\| \to 0 $.
Since $B$ is order-preserving, $\|B(x_k)\|_u \le \|B(y_k)\|_u \to 0$
with the latter holding by assumption.

So there exists some $m \in \N$ such
that $B (x_{k-1}) + 2^{-k} u  \le u$ for all $k \ge m$. Hence
\[
x_k = B (x_{k-1}) + 2^{-k} u  = y_k, \qquad k \ge m.
\]
In particular, $2^m x_m \ge u$ and $x_k \ge B (x_{k-1})$ for all $k \ge m$. Since $B$ is
order-preserving and  homogeneous,
\[
2^m x_{m +n} \ge B^n (2^m x_m ) \ge B^n (u)
\]
and
\[
2^m B (x_{m+n}) \ge B^{n+1} (u).
\]
Now $2^m B (x_{m+n}) \le (1/2) u$ for sufficiently large $n$. This shows
that, for some $n \in \N$, $ B^{n+1} (u) \le (1/2) u$ and $\|B^{n+1}\|_u \le 1/2$.
By Lemma \ref{re:Col-Wie-upp-loc},
 $\eta^u(B)= \inf_{n\in \N} \|B^n \|_u^{1/n}  < 1$, a contradiction.

 This shows that $x \ne 0$ and $B(x) \ge x$. Then $B^n (x) \ge x$
 for all $n \in \N$. By (\ref{eq:Col-Wie-num-low}),
 $[B^n]_x \ge 1$ and, by (\ref{eq:spec-rad-Col-Wie}), ${cw}(B) \ge 1
 = \eta^u(B) \ge \br_+^\sharp(B) \ge \br_{cw}(B)\ge cw(B)$.
\end{proof}

\begin{theorem}
\label{re:lattice-subeigen}
 Let the cone $X_+$ be a lattice and $B: X_+ \to X_+$ be
 order-preserving
and  homogeneous. Further let $u\in \dot X_+$ and
some power of  $B$ be monotonically $u$-compact
and  continuous
and some power of $B$
be uniformly $u$-bounded. Finally assume
that $r=\eta^u(B) > 0$.

Then there exists some $x \in \dot X_+$
such that $B (x) \ge r x$ and
$\eta^u(B)= \br_o^\sharp(B) = \br_{cw}(B)= cw(B)$.
\end{theorem}

\begin{proof} Replacing $B$ by $\frac{1}{\eta^u(B)} B$,
we can assume that $\eta^u(B) =1$.
Let $B^m$ be monotonically $u$-compact and $B^\ell$ be uniformly
$u$-bounded. Set $p= m + \ell$. Then $B^p$ is monotonically $u$-compact
and uniformly $u$-bounded and continuous.
By Theorem \ref{re:minihed-subeigen}, there exists some $w \in \dot X_+$
 such that $B^p (w)\ge  w$.
By Proposition \ref{re:subeigen-power-self}, there exists some $v \in \dot X_+$ such that $B(v) \ge v$.
\end{proof}

Recall the definition of $B$ being strictly increasing in
Definition \ref{def:stric-incr}.

\begin{theorem}
\label{re:lattice-eigen}
 Let  $X_+$  be a lattice and $B: X_+ \to X_+$ be monotonically continuous, strictly
increasing,
and  homogeneous. Further assume that a power of  $B$ is  monotonically $u$-compact
and some power is uniformly $u$-bounded
  for some $u \in \dot  X_+$.
Finally assume
that $r= \eta^u(B) > 0$.

Then there exists some $x \in \dot X_+$
such that $B (x) = r x$.
\end{theorem}

\begin{proof}
We can assume that $\eta_u(B) =1$.
By Theorem  \ref{re:minihed-subeigen}, there exists some $x \in X_+$,
$\|x\|=1$ such that $B (x) \ge  x$. Then the sequence $(x_n)_{n \in \Z_+}$
in $X_+$ defined by $x_n = B^n (x)$ is increasing.
The same proof as for Theorem \ref{re:strictly-power-comp} implies that
$\{\|x_n\|; n \in \N\}$ is a bounded set in $\R$.

Set $x_0 =x$. Then $x_n = B (x_{n-1})$ for $n \in \N$.
Since some power of $B$ is uniformly $u$-bounded, there exists some $c \ge 0$
and $m \in \N$
such that $x_n \le c \|x_{n-m}\| u = cu$ for $n \ge m$.
Since some power of $B$ is monotonically
$u$-compact and $(x_n) =(B^n(x))$ is increasing, a similar application
 of the convergence principle in Proposition \ref{re:monotone-converge}
 provides that $x_n \to y $ for some $y
\in X_+$ with
 $\|y\| =1$.
 Since $B$ is monotonically continuous, $\psi (B(y) - B(x_{n}))
\to 0$.
So $\psi(B(y) - x_{n+1} ) \to 0$. Then $\sharp y - B(y) \sharp
\le \sharp y - x_{n+1} \sharp + \sharp x_{n+1} - B(y) \sharp
= \psi(y -x_{n+1}) + \psi(B(y) - x_{n+1})
\to 0$. Thus $y = B(y)$.
\end{proof}

\begin{proposition}
\label{re:mon-comp-pert}
 Let $u \in \dot X_+$ and $B: X_+ \to X_+ \cap X_u$
be homogeneous and order-preserving. Further let
$B$ be monotonically $u$-compact, uniformly $u$-bounded and monotonically continuous.

 Let $\epsilon > 0$
and $\psi$ the  companion half-norm.
Set $B_\epsilon (x) = B(x) + \epsilon \psi (x) u$.

Then there exists
some $v \in \dot X_+$ such that $ B_\epsilon  (v)  = r_\epsilon v$
with $r_\epsilon = \eta^u(B_\epsilon)= CW_u(B_\epsilon)= \br_{cw}(B_\epsilon) =
cw(B_\epsilon)
>0$.
\end{proposition}

\begin{proof}
One readily checks that $B_\epsilon$ satisfies  the assumptions of
Theorem \ref{re:minihed-subeigen} and
$r_\epsilon = \eta^u (B_\epsilon) \ge \epsilon  > 0$.

We can assume that $r_\epsilon =1$. By Theorem \ref{re:minihed-subeigen},
 there exists some $w \in X_+$, $\psi(w)=1$,
such that $B_\epsilon w \ge w$. Let $w_n = B_\epsilon^n ( w)$ for $n \in \Z_+$. Then $(w_n)$ is
an increasing sequence in $X_+ \cap X_u$.  We claim that $\psi (w_n)  =1 $ for
all $n \in \N$. Suppose not. Then there exists some $n \in \N$ such that
$\psi (w_n) > \psi(w_{n-1})$. Then there exists some $\delta > 0$ such that
\[
B_\epsilon (w_n) \ge B_\epsilon (w_{n-1}) + \delta u.
\]
Since $B_\epsilon (w_{n-1}) \in X_u$, there exists some $\epsilon > 0$ such
that
\[
B_\epsilon (w_n) \ge (1+\epsilon) B_\epsilon (w_{n-1}).
\]
This implies $\eta^u(B_\epsilon) \ge 1+\epsilon$, a contradiction.
Since $B$ is uniformly $u$-bounded, it is also uniformly $u$-bounded with
respect to the  companion half-norm $\psi$ by Proposition
\ref{re:order-bounded-compan}. So there exists some $c >0$ such that
$w_n =  B (w_{n-1}) + \epsilon u \le c u$. Since $(w_n)$ is increasing
and $B_\epsilon$ is monotonically $u$-compact,
$w_{n+1} = B_\epsilon(w_n) \to v$ for some $v \in X_+$, $\psi(v) =1$. Since $B_\epsilon$
is monotonically continuous, $v = B_\epsilon (v)$ and $1 \le cw(B_\epsilon)
\le \br_{cw}(B_\epsilon) \le \eta^u(B_\epsilon)$. Since $v \ge \epsilon u$
and $B$ is uniformly $u$-bounded,
$v$ is $u$-comparable. This implies $CW_u(B_\epsilon) \le 1= \eta^u(B)$.
Since $CW_u(B_\epsilon) \ge \eta^u(B_\epsilon)$, we have equality.
\end{proof}

\begin{theorem}
\label{re:mon-comp-normal}
Let  $u \in \dot X_+$ and  $B: X_+ \to X_+ $
be homogeneous, order-preserving and continuous. Assume that
$B$ is uniformly $u$-bounded  monotonically $u$-compact.

 Then $\br_+(B) \ge   CW_u(B)= \eta^u(B)$
 with equality holding if $u$ is a normal point of $X_+$
 or a power of $B$ has the lower KR property. If $r:=CW_u(B) > 0$,
 there exists some $v \in \dot X_+$
 such that $B(v) \ge r v$.
\end{theorem}

\begin{proof} By Theorem \ref{re:Col-Wie-companion2},
$\br_+(B) \ge \gamma_B(u) \ge \eta^u(B)$
with equality holding if $u$ is a normal point of $X_+$  or some power of $B$
has the lower KR property.
By
(\ref{eq:upperCWbound}), $\eta^u(B) \le CW_u(B)$.

So, if $CW_u(B) =0$, the assertion holds, and we can assume that $CW_u(B)$ $ >0$.

Choose a sequence $(\epsilon_n)$ in $(0,1)$ with $\epsilon_n \to 0$.
Let $B_n:X_+ \to X_+$ be given by $B_n(x) = B(x) + \epsilon_n \psi(x)u$.
We combine  Proposition \ref{re:mon-comp-pert} and Theorem \ref{re:robust}
and obtain $\eta^u(B) = CW_u(B) \le
\br_+(B)$.

If $u$ is a normal point of  $X_+$ or some power
of $B$ has the lower KR property,  $\eta_u(B) = \br_+(B)$ which implies $CW_u(B)
= \br_+(B)$.

Assume that $r:=CW_u(B) > 0$. Then $\eta^u(B) = r >0$ and there
exists some $v \in \dot X_+$ with $B(v) \ge rv$ by Theorem \ref{re:minihed-subeigen}.
\end{proof}


\section{Application to a rank-structured population model
with mating}
\label{sec:model}


Let $X\subseteq \R^\N$ be an ordered normed vector space with
cone $X_+ = X \cap \R^\N_+$. Assume that the norm has the property
that $x_j \le \|x\|$ for all $x = (x_j) \in X_+$ and all $j \in \N$.
This implies that $X \subseteq \ell^\infty$ and $\|x\|_\infty \le \|x\|$
for all $x \in X$.

Define a map $B:X_+ \to \R^\N_+$,
$B(x) = (B_j(x))$, by
\begin{equation}
\left .
\begin{array}{rl}
B_1(x) =& q_1 x_1 + \displaystyle \sum_{j,k=1}^\infty  \beta_{jk} \min\{x_j, x_k\}
\\[4mm]
B_j(x) = &  \max\{ p_{j-1} x_{j-1},  q_j x_j \} , \qquad j \ge 2
\end{array}
\right\}
x= (x_j) \in X_+.
\end{equation}

Here $p_j, q_j \ge 0$,  $\beta_{j,k} \ge 0$
for all $j,k \in \N$.

The dynamical system $(B^n)_{n\in \N} $ can be interpreted
as the dynamics of a rank-structured population and, in a way,
is a discrete version (in a double sense) of the two-sex models
with continuous age-structure in
\cite{Had89, IMM}. $B_1(x)$
is the number of newborn individuals who all have the lowest
rank 1. Procreation is assumed to require some mating.
Mating is assumed to be rank-selective and
is described by taking the minimum of individuals in two
ranks.
The numbers $\beta_{jk}$ represent the fertility of a pair
 when the female has rank $j$ and the male has rank $k$
where a 1:1  sex ratio is assumed at each rank.
 The maps $B_j$, $j \ge 2$,
describe how individuals survive and move upwards in the ranks from
year to year where one cannot move by more than one rank
within a year. If  the population size is modeled in number of
individuals one would assume that $p_j, q_j \le 1$, but if
it is modeled in biomass, then such an assumption would not be made.
However, an individuals at rank $j$ is at least $j-1$ years old, and
so mortality eventually wins the upper hand such that the
assumption $p_j \to 0$ and $q_j \to 0$ is natural.

Since $B_j(x) \le (p_{j-1} + q_j) \|x\|$ for $j \ge 2$,
 $B$ is $u$-bounded with respect to
$u = (u_j)$ with
\begin{equation}
\label{eq:order-bound-choice}
u_1 =1, \qquad u_j = p_{j-1} + q_j, \quad j \ge 2,
\end{equation}
provided that $u \in X$.

If we choose $X = \ell^1$, it is sufficient  to assume that
\begin{equation}
\label{eq:pop-ell2}
 \sup_{k\in \N} \sum_{j=1}^\infty \beta_{jk} < \infty
\qquad \hbox{ or } \qquad
\sup_{j\in \N} \sum_{k=1}^\infty \beta_{jk} < \infty
\end{equation}

If $X = \ell^\infty$, it is sufficient to assume that
\begin{equation}
\label{eq:pop-ell3}
\sum_{j,k=1}^\infty \beta_{jk} < \infty.
\end{equation}
It does not appear that (\ref{eq:pop-ell3})
can be improved for $X = c_0$ or $X = bv$.

$B$ is compact on $X$ and uniformly $u$-bounded if

\[
\begin{cases} u \in \ell^p,  &  X = \ell^p, p \in [1,\infty)
\\
u \in c_0 , & X = c_0, c, \ell^\infty
\\
(p_j), (q_j) \in bv \cap c_0, & X = bv, bv \cap c_0.
\end{cases}
\]
Notice that, if $x^n = (x^n_j)$ is a bounded sequence
in $X$, then it is a bounded sequence in $\ell^\infty$
and, after a diagonalization procedure, has
a subsequence $(y^m)$ such that $(y^m_j)_{m\in \N}$ converges
for each $j \in \N$.

Apparently, $X= \ell^\infty, c, c_0$ require the weakest assumption
in terms of $u$.
It is not clear, however, whether an eigenvector in any of these
three spaces would also be in $bv$ if $(q_j), (p_j) \in bv \cap c_0$.
If $u \in \ell^1$, then $u$ is also in all the other spaces we have
considered, and $cw(B)= \br_{cw}(B) = \eta^u(B) = CW_u(B)$ in all spaces
with the radii and bounds not depending on the space. If $X$ is normal,
then also $\br_+(B) = cw(B)$ does not depend on the space. For $X= bv$,
we have  $\br_+(B) = cw(B)$ as well provided that  $(p_j), (q_j) \in bv \cap c_0$ such that
$B$ is compact also in $bv$ or in $bv \cap c_0$.

To derive estimates for the cone spectral radius of $B$, one runs into algebraic
difficulties very soon except when attempting $cw(B)$. Let $x^n$ be the sequence
where the first $n$ terms are 1 and all others zero. Then,
\[
\begin{split}
[B]_{x^1} = &
q_1 + \beta_{11},
\\
[B]_{x^m} = & \min \Big \{ q_1 + \sum_{j,k=1}^m \beta_{j,k}, \; \inf_{j=2}^m \max \{p_{j-1}, q_j \} \Big \}
, \quad m \ge 2.
\end{split}
\]
Let $e^m$ be the sequence where the $m^{th}$ term is one and all other terms zero.
Then
\[
\begin{split}
[B]_{e^1} = &
q_1 + \beta_{11},
\\
[B]_{e^m} = & q_m, \qquad m \ge 2.
\end{split}
\]
Since $cw(B)$ is an upper bound for all these numbers, we obtain the
following two conditions for $cw(B)>0$ each of which is sufficient:

\begin{quote}
Condition 1: $q_1 + \beta_{11} > 0$ or $q_j > 0$ for some $j \in \N$, $j >1$.
\end{quote}

\begin{quote}
Condition 2: For some $j,k \in \N$ and $m= \max\{j,k\}$, we have $\beta_{jk} > 0$
and $p_i > 0 $ for $i=1,\ldots, m-1$ .
\end{quote}

Actually, at least one of these conditions must be satisfied for   $\br_+(B) >0$
to hold.
If condition 1 fails, then
\begin{equation}
\label{eq:mate-special}
\left .
\begin{array}{rl}
B_1(x) =&  \displaystyle \sum_{j,k=1}^\infty  \beta_{jk} \min\{x_j, x_k\}
\\[4mm]
B_j(x) = &   p_{j-1} x_{j-1} , \qquad j \ge 2
\end{array}
\right\}
x= (x_j) \in X_+.
\end{equation}

If $r=  \br_+(B)=cw(B) > 0$, then there exists some $x \in \dot X_+$ with
$(1/r) B(x) = x$. $(1/r) B$ corresponds to the map where all parameters are
divided by $r$. This division does not affect our assumptions and conditions.
So we can assume that $B(x)=x$, $x \in \dot X_+$. By (\ref{eq:mate-special}),
\[
x_j = p_{j-1} x_{j-1}, \quad j \ge 2.
\]
By recursion,
\[
x_j = \prod_{i=1}^{j-1} p_j x_1, \qquad j \ge 2.
\]
Since $x_1 = \sum_{j,k=1}^\infty \beta_{jk} \min \{x_j, x_k\}$, failure of condition 2
implies $x_1=0$ which implies $x=0$. This contradiction tells us that $\br_+(B) =0$
if both condition 1 and 2 fail.

In order to find estimates of $\br_+(B)$ from above,
we can use Corollary \ref{re:radius-estimate-above}.

Take the sequence $e$
where all terms are 1 for $X= \ell^\infty$ or $X = bv$ provided that
$\sum_{j,k=1}^\infty \beta_{jk} < \infty$.  Then
\[
\br_+(B) \le \eta^e (B) \le \|B(e)\|_e = \max \Big \{ q_1 + \sum_{j,k=1}^\infty \beta_{jk}, \sup_{m \ge 2} \max \{p_{m-1},
q_m\} \Big \}.
\]
We obtain this estimate also for $X = c_0, bv\cap c_0, \ell^1$
provided that there exist  $\theta \in (0,1)$ and $c > 0$ such that
$\max \{p_{m-1}, q_m\} \le c\theta^m $ for all $m \ge 2$.
Then $B$ is uniformly $w$-bounded for $w = (\theta^n)$ and
\[
\br_+(B) \le  \|B\|_w = \max \Big \{ q_1 +
\sum_{j,k=1}^\infty \beta_{jk} \theta^{j+k -1}, \sup_{m \ge 2} \max \{p_{m-1}/\theta,
q_m\} \Big \}.
\]
This estimate also holds for all $\tilde \theta \in (\theta, 1)$, and so
we can take the limit for $\theta \to 1$ and obtain the previous estimate.

We obtain another estimate if there exist
$\alpha >1$ and $c > 1$ such that
$\max \{p_{m-1}, q_m\} \le m^{-\alpha} $ for all $m \ge 2$.
Then $B$ is uniformly $w$-bounded for $w = n^{-\alpha}$ and
\[
\br_+(B) \le \|B\|_w = \max \Big \{ q_1 + \sum_{j,k=1}^\infty \beta_{jk} j^{-\alpha} k^{-\alpha}, \sup_{m \ge 2} \max \{p_{m-1} (m/m-1)^\alpha , q_m\} \Big \}.
\]
This estimate also holds for all $\tilde \alpha \in (1, \alpha)$, so we can take
the limit $\alpha \to 1$.

As an illustration of Theorem \ref{re:attractor}, we assume that
$\beta_{jk} $, $p_j$ and $q_j$ are functions of $x \in X_+$ and
let $F(x) $ be $B(x)$ with this functional dependence.

We assume that $p_{j-1}(x) + q_j(x) \to 0$ uniformly for $j \in \N$
as $\|x\| \to \infty$, further $\sum_{j,k=1}^\infty \beta_{jk}(x) \to 0$
as $\|x\| \to \infty$. Let $\epsilon > 0$ still to be chosen  and choose $c > 0$
such that
\[
\left.
\begin{array}{c}
p_{j-1}(x) + q_j(x) \le \epsilon
\\
\displaystyle \sum_{j,k=1}^\infty \beta_{jk}(x) \le \epsilon
\end{array}
\right \} \qquad \|x\| \ge c.
\]
Choose $A: X \to X $, $A(x) = (A_j(x))$,
\begin{equation}
\left .
\begin{array}{rl}
A_1(x) =& \epsilon x_1 + \displaystyle \sum_{j,k=1}^\infty  \tilde \beta_{jk} x_j
\\[4mm]
A_j(x) = &  \epsilon (x_{j-1} + x_j) , \qquad j \ge 2
\end{array}
\right\}
x= (x_j) \in X_+.
\end{equation}
Here $\tilde \beta_{jk} = \sup_{\|x\| \ge c} \beta_{jk}(x)$.
Then $B(x) \le A(x)$ whenever $\|x\| \ge c$. Choose $\epsilon < 1/2$ if $X = \ell^\infty$
and $\epsilon < 1/3$ if $X =\ell^1$. Then $\|A\| < 1$ and $F$ is point-dissipative
by Theorem \ref{re:attractor}.

To make $F$ compact in $X =\ell^1$, we
assume that for any $c > 0$ there exists $u \in \ell^1$ such that  $\sup_{k \ge 1}\sum_{j=1}^\infty \beta_{jk}(x)
\le u_1$ and $p_{j-1}(x) + q_j(x) \le u_j$ for $j \ge 2$ and all $x \in X_+$ with $\|x\| \le c$. Under these assumptions, the discrete semiflow $(F^n)$ has a compact attractor
of bounded sets.

\bigskip

\noindent
{\bf Acknowledgement.}
I thank Wolfgang Arendt, Karl-Peter Hadeler,  and Roger Nussbaum for useful
hints and comments.

 \bibliography{}

\end{document}